\documentclass[a4paper,11pt]{article}
\usepackage[hmargin={25mm,25mm},vmargin={25mm,30mm}]{geometry}
\usepackage{amsmath,amsfonts,amsthm,amssymb,mathtools}
\usepackage{cancel}
\usepackage{comment}
\usepackage{enumerate}
\usepackage{graphicx}
\usepackage[ocgcolorlinks,allcolors={blue}]{hyperref}
\usepackage{lmodern}
\usepackage{hyperref}

\usepackage{xcolor}

\usepackage{authblk}
\usepackage[percent]{overpic}
\usepackage{epstopdf}

\usepackage{tabularx}
\usepackage{multirow}
\usepackage{booktabs}
\usepackage{bm}
\mathtoolsset{showonlyrefs} 

\usepackage{tcolorbox}

\usepackage{bbold}

\newtheorem{theorem}{Theorem}
\newtheorem{proposition}[theorem]{Proposition}
\newtheorem{lemma}[theorem]{Lemma}
\newtheorem{corollary}[theorem]{Corollary}
\theoremstyle{remark}
\newtheorem{remark}[theorem]{Remark}
\theoremstyle{definition}

\newcommand{\uu}{\mathbf{u}}

\usepackage{stackengine}

\usepackage{newtxtext,newtxmath}

\newcommand{\ale}[1]{{\color{red}#1}}
\usepackage{soul}

\usepackage[normalem]{ulem}
\newcounter{corr}
\definecolor{violet}{rgb}{0.580,0.,0.827}
\newcommand{\corr}[3]{\typeout{Warning : a correction remains in page \thepage}
\stepcounter{corr}        
{\color{blue}\ifmmode\text{\,\sout{\ensuremath{#1}}\,}\else\sout{#1}\fi}{\color{red}#2}{\color{violet} [#3]}
}
\definecolor{mygreen}{rgb}{0,0.4,0.2}

\newcommand{\Pol}{\mathbb{P}}

\newcommand{\PiNablaE}{\Pi^{\nabla,E}_k}
\newcommand{\PiNablaF}{\Pi^{\nabla,F}_k}
\newcommand{\PiZeroE}{\Pi^{0,E}_k}
\newcommand{\PiZeroF}{\Pi^{0,F}_k}
\newcommand{\email}[1]{\href{mailto:#1}{#1}}

\usepackage[dvipsnames]{xcolor}
\usepackage{xstring}

\newcounter{colorindex}
\newcounter{charpos}

\newcommand{\rainbow}[1]{%
  \StrLen{#1}[\textlength]
  \ifnum\textlength>0
    \setcounter{colorindex}{0}%
    \setcounter{charpos}{1}%
    \loop
      \StrChar{#1}{\value{charpos}}[\currentchar]%
      \IfStrEq{\currentchar}{ }{ }{%
        \stepcounter{colorindex}%
        \ifcase\value{colorindex}%
          \or\textcolor{red}{\currentchar}%
          \or\textcolor{orange}{\currentchar}%
          \or\textcolor{yellow}{\currentchar}%
          \or\textcolor{green}{\currentchar}%
          \or\textcolor{blue}{\currentchar}%
          \or\textcolor{violet}{\currentchar}\setcounter{colorindex}{0}%
        \fi
      }%
    \ifnum\value{charpos}<\textlength
      \stepcounter{charpos}%
    \repeat
  \fi
}

\title{Conforming/Non-conforming Virtual Elements and application to elasticity problems in curved three-dimensional domains}

\author[1,2]{L. Beir\~{a}o da Veiga\footnote{\email{lourenco.beirao@unimib.it}}}
\author[1]{F. Dassi \footnote{\email{franco.dassi@unimib.it}}}
\author[1,2]{A. Russo \footnote{\email{alessandro.russo@unimib.it}}}
\author[1]{M. Trezzi \footnote{\email{manuel.trezzi@unimib.it}}}

\affil[1]{Dipartimento di Matematica e Applicazioni, 
Universit\`a degli Studi di Milano-Bicocca, 
Via Roberto Cozzi 55  - 20125 Milano, Italy}

\affil[2]{CNR-IMATI, Via Ferrata 1 - 27100 Pavia, Italy}

\begin{document}

\maketitle

\section{Introduction} 

The Virtual Element Method (VEM) is a fairly recent methodology for the discretization of problems in Partial Differential Equations and, since its inception in 2013 \cite{BeiradaVeiga-Brezzi-Cangiani-Manzini-Marini-Russo:2013, BeiradaVeiga-Brezzi-Marini-Russo:2014} enjoyed a wide success in the mathematical and engineering communities. The main trademark of VEM is its flexibility in dealing with meshes composed of general (even non-convex or not simply connected) polytopes, yielding potential advantages in handling complex geometries, grid refinement and coarsening, moving meshes, high distortions, cracks, multiple dimensional domains, etc.  

In the mathematical community of Finite Elements (FEM), when dealing with problems that are naturally conforming in $H^1$, two prominent and distinct approaches are using conforming or non-conforming elements. The first choice is the most classical and has the advantage of yielding a solution which is in the natural variational space and globally continuous, but the second one is also appealing due to specific properties, related for instance to robustness to locking and conservation properties.

In the present article the flexibility of VEM is exploited in a novel way. We develop a Virtual Element Method of general order which is conforming/non-conforming, in the sense that the user is free to choose which faces of the mesh are to be treated as conforming (the solution will glue continuously across such faces) and which are to be treated as non-conforming. The ensuing local VEM space is built through an ad-hoc definition exploiting local (element-wise) problems with mixed Dirichlet/Neumann boundary conditions. Such a construction would be essentially impossible with standard FEM; therefore, even when using standard tetrahedral grids, the resulting method fundamentally differs from classical FEM.

After defining the element, its degrees of freedom and the computability of the associated projections, we showcase a model development for the three dimensional linear elasticity problem. We develop the method and prove its convergence properties, which requires a stability analysis of the involved discrete forms and interpolation bounds for the novel spaces here proposed.

Afterwards, as an example of the potential utility of such VEM spaces, we present an application for 3D problems (still in the realm of linear elasticity) posed on curved domains with exact geometry approximation. Indeed it must be noted that, while the VEM for 2D problems with curved boundaries is nowadays quite developed (see for instance \cite{BeiraodaVeiga-Russo-Vacca:2019, BeiraodaVeiga-Brezzi-Marini-Russo:2020, Aldakheel2020, art-bdv-das, wriggers2020virtual, Dassi-Fumagalli-Losapio-Scialo-Scotti-Vacca,Dassi-Fumagalli-Mazzieri-Scotti-Vacca,da2024c} or \cite{Botti-DiPietro:2018,Yemm:2023,Warburton2013,Chan2017,Kawecki2020} for other related technologies), VEM on curved domains in three space dimensions still encounter major challenges. Possible approaches are shown in \cite{curved3D}, where a peculiar master-slave construction is required, and \cite{bertoluzza2022weakly}, where the boundary conditions need to be enforced weakly with a suitable technique. None of the above, although very interesting, contributions makes simply use of a primal Virtual Element Space, but introduce specific constructions to accomodate for the curved boundary. 
Another approach is that in \cite{DASSI20221}, but where a mixed formulation with discontinuous discrete variables is adopted.

Another possible answer was shown in \cite{curvo-Yi,Liu2025}, but making use of non-conforming VEM: as a consequence, when such idea is applied in the context of solid mechanics, the ensuing displacement solution will be non-continuous, which is non-physical. In the present contribution, by exploiting the idea described above, we propose a methodology which uses curved non-planar and non-conforming faces only on the boundary of the 3D domain and planar (typically polygonal) conforming faces inside the domain. In such a way, we are able to deal with curved geometries and deliver a continuous displacement solution. 

We test this proposed approach with a series of numerical benchmarks, including (1) standard straight geometries, (2) curved geometries with exact mesh approximation and (3) curved geometries with patch-wise, non-piecewise planar, mesh approximation. The results are encouraging and in line with the theoretical predictions.

The paper is organized as follows. After presenting the novel methodology in Section \ref{sec:1}, we develop its application for linear elasticity in Section \ref{sec:2}, including theoretical derivations such as error bounds. In Section \ref{sec:3} we detail its extension to curved domains. Finally, in Section \ref{sec:4} we develop the numerical tests. 

\section{The C-NC Virtual Element Space on standard polyhedral meshes}\label{sec:1}

\subsection{Basic notations and mesh assumptions}\label{sec:mesh-ass}

For a generic measurable set $D$ and $m \geq 0$, we denote as usual the Hilbert space $H^m(D)$ as the standard Sobolev space of square integrable functions with square integrable derivatives up to order $m$. We denote by $\| \cdot \|_{m,D}$ and $| \cdot |_{m,D}$ the standard norm and seminorm associated with $H^m(D)$, respectively. In particular, for $m=0$, $H^0(D)$ coincides with $L^2(D)$, and we denote its norm as $\| \cdot \|_{0,D}$. Given a mesh partition $\Omega_h$ of the domain $\Omega$, we define the broken global space as
\[
H^m(\Omega_h) \coloneqq \{ v \in L^2(\Omega) \mid v|_E \in H^m(E) \quad \forall E \in \Omega_h \}
\]
equipped with the broken global norm and seminorm
\[
\| v \|^2_{m,\Omega_h} \coloneqq \sum_{E \in \Omega_h} \| v \|^2_{m,E} \,,
\qquad
| v |^2_{m,\Omega_h} \coloneqq \sum_{E \in \Omega_h} | v |^2_{m,E} \,.
\]
Moreover, the space of polynomials up to degree $k$ on a geometrical entity $D$ is denoted with $\mathbb P_k(D)$. \medskip 

We consider a family of meshes $\{\Omega_h\}_h$, which are conforming partitions (in the standard Finite Element sense) of the domain $\Omega$ into polyhedrons. We denote by $E$ a generic element of the mesh, by $F$ a generic face, by $e$ a generic edge, and by $\nu$ a generic vertex.
Let $D$ be an element, a face, or an edge, we denote with $h_D$ its diameter, and with $|D|$ its size.
We make the following standard VEM assumption: it exists a constant $\rho$ independent of $h$ such that: 
\begin{itemize}
\item each polyhedron $E$ and each face $F$ are star-shaped with respect to a ball of radius greater than or equal to $\rho h_E$ and $\rho h_F$, respectively;
\item $h_F \geq \rho h_E$ holds for every face $F$.
\end{itemize}

\subsection{Conforming Virtual Element Spaces on faces}\label{sec:facevem}
Before introducing the C-NC VEM space in 3D, we first recall the definition of the 2D ``enhanced'' nodal VEM space. 
Indeed, the construction of the 3D VEM space \cite{Ahmad-Alsaedi-Brezzi-Marini-Russo:2013, BeiradaVeiga-Brezzi-Marini-Russo:2014} requires the preliminary definition of a VEM space on the faces of the mesh. 
Note that this step is not required for the nonconforming VEM space \cite{AyusodeDios-Lipnikov-Manzini:2016}. 
Let $F$ be a face in $\Omega_h$ (i.e. a planar polygon); we define the space
\begin{equation}\label{eq:vem-space-face}
    \begin{aligned}
    V_h^k(F) := \Big\{ v_h \in C^0(F) \cap H^1(F) \, : \, & v_h|_e \in \Pol_k(e) \ \forall e \subset \partial F 
    \, , \Delta v_h \in \Pol_{k}(F) \, , \\
    & \int_F (v_h-\PiNablaF v_h) p_k\,{\rm d}F = 0 \ \forall \ p_k \in \Pol_k/\Pol_{k-2}(F) \Big\} \, ,
\end{aligned}
\end{equation}
where the Laplace differential operator $\Delta$ is intended as planar, 
$\Pol_k/\Pol_{k-2}(F)$ denotes the space spanned by the scaled monomials of degree $k$ and $k-1$, and
the operator $\PiNablaF : H^1(F) \rightarrow \Pol_k(F)$ is the classical one defined in \cite{BeiradaVeiga-Brezzi-Cangiani-Manzini-Marini-Russo:2013}:
$$
\left\{
\begin{aligned}
&\int_F \nabla (I - \PiNablaF) v_h \cdot \nabla p_k \,{\rm d}F = 0 \qquad \forall v_h \in V_h^k(F), \, \forall p_k\in\Pol_k(F), \\
&P_0(v_h - \PiNablaF) = 0 \qquad \forall v_h \in V^k_h(F),
\end{aligned}
\right .
$$
where $P_0: V^k_h(F) \rightarrow \mathbb{R}$ is a projection onto the space of constants, which is necessary to uniquely determine the projection.
Following \cite{BeiradaVeiga-Brezzi-Marini-Russo:2014}, we make the following choice for $P_0$:
\begin{equation} \label{eq:P0}
P_0(v_h) :=
\int_{\partial F} v_h \,{\rm d}s \qquad \forall v_h \in V_h^k(F) \, .
\end{equation}
As shown in \cite{Ahmad-Alsaedi-Brezzi-Marini-Russo:2013}, a set of degrees of freedom for the space \eqref{eq:vem-space-face} is given by:

\smallskip\noindent
{\bf Conforming face DoFs} ($v_h \in V_h^k(F)$):
\begin{itemize}
\item pointwise evaluation $v_h(\nu)$ at all $\nu \in \partial F$ ;
\vspace{-2mm}
\item pointwise evaluation of $v_h$ at $k-2$ distinct points of each $e \subset \partial F$;
\vspace{-2mm}
\item moments $\frac{1}{|F|}\int_F v_h \, \varphi_i$, $i=1,..,M$ with $\{ \varphi_i\}_{i=1}^M$ basis of $\Pol_{k-2}(F)$ with $\|\varphi_i \|_{L^\infty(F)} = 1 $.
\end{itemize}
We highlight that for $k=1$, the last two sets of Dofs are empty.
\begin{remark}\label{mimikyu}
It can be easily shown that the $L^2(F)$ projection operator $\Pi_E^{0,k} : V_h^k(F) \rightarrow \Pol_k(F)$ defined as
\[
\int_F (I - \PiZeroF) v_h \, p_k \,{\rm d}F = 0 \qquad \forall v_h \in V_h^k(F), \, \forall p_k\in\Pol_k(F),
\]
is computable only using the knowledge on the above degrees of freedom values and the enhancing technique \cite{BeiradaVeiga-Brezzi-Marini-Russo:2014}. Note that $\PiNablaF$ and $\PiZeroF$ coincide for $k=1$.
\end{remark}

\subsection{The C-NC Virtual Element Space}\label{sec:NC-space}
Given a polyhedron $E \in \Omega_h$ we split its boundary faces into two subsets: conforming ${\cal F}_E^c$ and non-conforming ${\cal F}_E^n$. We denote by $\Gamma^c_E$ the element boundary region given by the closure of the union of all faces in ${\cal F}_E^c$, and by $\Gamma^n_E = \partial E / \Gamma^c_E$. 
We define the space
\begin{equation}\label{eq:c-nc-space}
\begin{aligned}
W_h^k(E) := \Big\{ v_h \in H^1(E) \, & :\, \Delta v \in \Pol_{k-2}(E) , 
\ v_h|_{\Gamma^{c}_E} \in C^0(\Gamma^{c}_E) \, ,\\
&v|_F \in V_h^k(F) \ \forall \ F \in {\cal F}^c_E \, , \\
&(\partial_{{\bf n}} v)|_F \in \Pol_{k-1}(F) \ \forall \ F \in {\cal F}^n_E  \Big\} \, ,
\end{aligned}
\end{equation}
where $\partial_{{\bf n}}$ denotes the outward normal derivative to the face.

The above space is well defined, being associated to a Poisson problem on the domain $E$ with Dirichlet (on $\Gamma^c_E$) and Neumann (on $\Gamma^n_E$) boundary conditions. Note that, on the faces of ${\cal F}^c_E$ the definition follows that of classical nodal VEM \cite{BeiraodaVeiga-Brezzi-Marini:2013,beirao2016virtual}, while on the faces in ${\cal F}^n_E$ the definition corresponds to the non-conforming VEM (\cite{AyusodeDios-Lipnikov-Manzini:2016,butterflypaper}). In particular, if all faces are in the first set we recover the classical conforming VEM, while if all faces are in the second set we recover the non-conforming VEM. The flexibility of VEM allows also to develop this ``hybrid'' element.

We introduce the following set of degrees of freedom for $W^k_h(E)$, where all the involved polynomial basis functions are assumed to have unitary $L^\infty$ norm on the element/face/edge.

\smallskip\noindent
{\bf Degrees of freedom} ($v_h \in W_h^k(E)$):
\begin{itemize}
\item pointwise evaluation $v_h(\nu)$ at all vertexes $\nu$ of $E$ which lay in $\Gamma^c_E$;
\vspace{-2mm}
\item pointwise evaluation of $v_h$ at $k-2$ distinct points for each edge $e$ of $E$ which lays in $\Gamma^c_E$;
\vspace{-2mm}
\item moments $\frac{1}{|F|} \int_F v_h \, \varphi_i$, $i=1,..,M$ with $\{ \varphi_i\}_{i=1}^M$ basis of $\Pol_{k-2}(F)$ for all $F \in {\cal F}^c_E$ ;
\vspace{-2mm}
\item moments $\frac{1}{|F|} \int_F v_h \, \varphi_i$, $i=1,..,K$ with $\{ \varphi_i\}_{i=1}^K$ basis of $\Pol_{k-1}(F)$ for all $F \in {\cal F}^n_E$ ;
\vspace{-2mm}
\item moments $\frac{1}{|E|} \int_E v_h \, \varphi_i$, $i=1,..,L$ with $\{ \varphi_i\}_{i=1}^L$ basis of $\Pol_{k-2}(E)$.
\end{itemize}

\begin{proof}
We consider the case in which ${\cal F}^c_E$ is non-void, the other case being the standard non-conforming VEM \cite{AyusodeDios-Lipnikov-Manzini:2016}. We start observing that, due to the well-posedness of the problem defining $V_h^k(E)$, the dimension of $V_h^k(E)$ corresponds to the total dimension of the (volume and boundary) data; an easy count shows that the dimension of such space corresponds to the total number of DoFs here above.
Let now $v_h \in W_h^k(E)$ and assume that all degrees of freedom associated with $v_h$ vanish.
For any face $F \in {\cal F}^c_E$, the restriction $v_h|_F$ belongs to the nodal conforming virtual element space $W_h^k(F)$ defined in \eqref{eq:vem-space-face}. The degrees of freedom defined on $\Gamma^c_E$ constitute a unisolvent set for $W_h^k(F)$ \cite{BeiradaVeiga-Brezzi-Cangiani-Manzini-Marini-Russo:2013}. Since these degrees of freedom vanish, we conclude that $v_h|_F \equiv 0$ for all $F \in {\cal F}^c_E$. Consequently, $v_h \equiv 0$ on $\Gamma^c_E$.
Thanks to integration by parts, and the fact that $v_h$ is zero on the conforming faces, we have that
\[
|\nabla v_h|^2_{1,E} = \int_E \nabla v_h \cdot \nabla v_h \, {\rm d}E = - \int_E v_h \, \Delta v_h \, {\rm d}E + \sum_{F \in {\cal F}_E^n} \int v_h \, (\partial_{\bf n} v_h) {\rm d} s \, .
\]
The first integral is zero since the moments up to degree $k-2$ on $E$ are zero.
Regarding the boundary integral, we recall that, for any $F \in {\cal F}_E^n$, it holds 
$\partial_{\bf n} v_h \in \Pol_{k-1}(F)$ and the integral of $v_h$ against polynomials of degree up to degree $k-1$ vanishes. 
Hence, the boundary integral is zero and we have
\[
| \nabla v_h|^2_{1,E} = 0 \, ,  
\]
yielding that the function is constant. Since $v_h$ is zero on the conforming faces, we conclude that $v_h \equiv 0$. 
\end{proof}
Similarly to the 2D case on faces, we introduce the projection $\PiNablaE : H^1(E) \rightarrow \Pol_k(E)$: 
\begin{equation}
\left\{
\begin{aligned}
&\int_E \nabla (I - \PiNablaE) v \cdot \nabla p_k \,{\rm d}E = 0 \quad \forall p_k\in\Pol_k(E) \, , \\
&  \int_{\partial E} (I - \PiNablaE) v \,{\rm d}s   = 0 \, ,
\end{aligned}
\right .
\label{eqn:scalPiNabla}
\end{equation}
here defined for any $v \in H^1(E)$.

On the space $W_h^k(E)$, the projection $\PiNablaE$ is explicitly computable from the DoF values. Indeed, we only need to distinguish between conforming and nonconforming faces; we have that
\begin{equation}\label{eq:squirtle}
\begin{aligned}
\int_E \nabla v_h \cdot \nabla p_k \, {\rm d}E 
&= 
-\int_E v_h \, \Delta p_k \, {\rm d}E
+ 
\int_{\partial E} (\partial_\mathbf{n} p_k) v_h \, {\rm d}s
\\
&
= -\int_E v_h \, \Delta p_k \, {\rm d}E
+ 
\sum_{F \in {\cal F}_E^c}
\int_{F} (\partial_\mathbf{n} p_k) \PiZeroF v_h \, {\rm d}F
+ 
\sum_{F \in {\cal F}_E^{nc}}
\int_{\partial F} (\partial_\mathbf{n} p_k) \Pi^{0,F}_{k-1} v_h \, {\rm d}s \, .
\end{aligned}
\end{equation}
We observe that the projection $\PiZeroF$ on the conforming faces is computable due to Remark \ref{mimikyu} and the projection $\Pi^{0,F}_{k-1}$ on the nonconforming faces is computable using the fourth group of DoFs in the above list.
By the same arguments, we can show that also the operator ``projection $\Pi^{0,E}_{k-1}$ of the gradient'', defined as
\begin{equation}\label{eq:pizerograd}
\int_E (I - \Pi^{0,E}_{k-1})\nabla v \cdot \mathbf p_k \, {\rm d}E 
= 
0 \qquad \forall \mathbf p_k \in [\mathbb{P}_{k-1}(E)]^{3} \quad \forall v \in H^1(E) \, ,
\end{equation}
is computable (through the DoF values) for functions in $W_h^k(E)$.
We introduce the enhanced version of the C-NC VEM space by employing the standard technique:
\begin{equation}\label{pokeball}
\begin{aligned}
V_h^k(E) := \Big\{ v_h \in H^1(E)\, & :\, \Delta v_h \in \Pol_{k}(E) \, , \, 
v_h|_{\Gamma^{c}_E} \in C^0(\Gamma^{c}_E) \, , \\
& \int_E (v_h-\PiNablaE v_h) p_k \,{\rm d}E = 0 \ \forall \ p_k \in \Pol_k/\Pol_{k-2}(E) \, , \\
&v_h|_F \in V_h^k(F) \ \forall \ F \in {\cal F}^c_E \, , \,
(\partial_{{\bf n}} v_h)|_F \in \Pol_{k-1}(F) \ \forall \ F \in {\cal F}^n_E  \Big\} .
\end{aligned}
\end{equation}
By standard VEM arguments, it can be shown that (1) the DoFs of the original space remain valid also for $V_h^k(E)$ and (2) the enhanced space $V_h^k(E)$ allows for the computation of the $L^2$ orthogonal projection on polynomials of degree $k$:
\begin{equation}
\int_E (I - \PiZeroE) v_h \, p_k \, {\rm d} E = 0 \quad \forall p_k \in \mathbb{P}_k(E) \, .
\label{eqn:scalPiZero}
\end{equation}

Finally, the {\bf global conforming-nonconforming VEM space}  $V_h^k(\Omega_h)$ will be built as follows. Given the polyhedral mesh $\Omega_h$, we split the faces into conforming (set ${\cal F}^c$) and non-conforming (set ${\cal F}^n$) faces. Such subdivision will automatically induce, for each element $E \in \Omega_h$, the sets ${\cal F}_E^c$ and ${\cal F}_E^n$, and such information is sufficient to define $V_h^k(E)$, see \eqref{pokeball}. All DoF values on faces that are in common to two elements are, as usual, assumed to be single-valued. The final space will be defined, following the standard Finite Element assembling procedure, as the space of functions that
(1) when restricted to each $E \in \Omega_h$ are in $V_h^k(E)$ and 
(2) which glue, across faces and edges, accordingly to the common DoFs. 
Note that the ensuing functions are locally in $H^1$ and glue continuosly (only) across faces in ${\cal F}_E^c$.
Such construction also naturally defines the global DoFs.

\begin{remark}
Although we have presented the C-NC VEM space in three dimensions, the definition can be trivially adapted to the 2D case. Let $E \in \Omega_h$ be a polygon and let us split its boundary edges into a conforming set ${\cal E}_E^c$ and a non-conforming set ${\cal E}_E^n$. We denote by $\Gamma^c_E$ the union of the edges in ${\cal E}_E^c$. The two-dimensional space is defined as
\begin{equation}
\begin{aligned}
W_h^k(E) := \Big\{ v_h \in H^1(E) \, & :\, \Delta v_h \in \Pol_{k-2}(E) \, , \, 
v_h|_{\Gamma^{c}_E} \in C^0(\Gamma^{c}_E) \, , \\
&v_h|_e \in \Pol_k(e) \ \forall \ e \in {\cal E}^c_E \, , \,
(\partial_{{\bf n}} v_h)|_e \in \Pol_{k-1}(e) \ \forall \ e \in {\cal E}^n_E \Big\} \, ;
\end{aligned}
\end{equation}
also the enhanced version $V_h^k(E)$ can be constructed, analogously to the 3D case. 
In this setting, the boundary conditions on the conforming edges simplify to standard polynomial continuity, recovering the classic VEM approach on $\Gamma^c_E$ and the non-conforming VEM approach on the remaining boundary.
\end{remark}

\begin{remark}
We note that on conforming faces $F \in {\cal F}^c$ one could also introduce the Serendipity VEM procedure to strongly reduce the number of associated face DoFs, see for instance \cite{sere1,sere2}.
\end{remark}

\section{Application to linear elasticity}\label{sec:2}

In this section, we apply the C-NC VEM to a model problem. Specifically, we focus on the linear elasticity model problem.

\subsection{The continuous problem}

Let $\Omega$ be an elastic polyhedral domain, with boundary $\partial \Omega$ partitioned into two disjoint sets $\Gamma_N, \Gamma_D$, the latter with positive measure.
We consider the following model problem
\begin{equation}\label{eq:problemacontinuo}
\begin{cases}
-\nabla \cdot \boldsymbol{\sigma}(\mathbf{u}) = \mathbf{f} & \text{in } \Omega, \\
\mathbf{u} = \mathbf{g} & \text{on } \Gamma_D, \\
\boldsymbol{\sigma}(\mathbf{u}) \mathbf{n} = \mathbf{h} & \text{on } \Gamma_N,
\end{cases}
\end{equation}
where $\mathbf{u}$ is the displacement, the stress tensor $\boldsymbol{\sigma}$ is defined as
\[
\boldsymbol{\sigma}(\mathbf{u}) = 2\mu \boldsymbol{\varepsilon}(\mathbf{u}) + \lambda \text{tr}(\boldsymbol{\varepsilon}(\mathbf{u})) \mathbf{I} \, ,
\]
while $\lambda $ and $\mu$ are the Lam\'e coefficients, and  $\boldsymbol{\varepsilon}(\mathbf{u}) = \frac{1}{2} \left( \nabla \mathbf{u} + \nabla \mathbf{u}^T \right)$ denotes the symmetric part of the gradient. 
Let $\mathbf{V}_\mathbf{g} \coloneqq \{\mathbf{v}\in [H^1(\Omega)]^3$ s.t. $\mathbf{u} = \mathbf g$ on $\Gamma_D$\}.
The variational problem reads
\begin{equation}
\text{find } \mathbf{u} \in \mathbf{V}_\mathbf{g} \text{ such that: }
a(\mathbf{u}, \mathbf{v}) = F(\mathbf{v}) \quad \forall \mathbf{v} \in \mathbf{V}_\mathbf{0} \, ,
\end{equation}
where the bilinear form $a(\cdot, \cdot)$ and the linear functional $F(\cdot)$ are defined as:
$$
a(\mathbf{u}, \mathbf{v}) = \int_{\Omega} \boldsymbol{\sigma}(\mathbf{u}) : \boldsymbol{\varepsilon}(\mathbf{v}) \, {\rm d} \Omega \, , \qquad
F(\mathbf{v}) = \int_{\Omega} \mathbf{f} \cdot \mathbf{v} \, {\rm d} \Omega + \int_{\Gamma_N} \mathbf{h} \cdot \mathbf{v} \, {\rm d} \Gamma \, .
$$
The well-posedness of this problem relies on the Korn's inequality
\begin{equation}
\|\boldsymbol{\varepsilon}(\mathbf{v})\|_{0,\Omega} \geq C_K |\mathbf{v}|_{1,\Omega} \quad \forall \mathbf{v} \in \mathbf{V}_\mathbf{0} \, ,
\end{equation}
and the Lax-Milgram lemma.

\subsection{The discrete problem}\label{sec:elast:discr}
We define the global virtual space 
\[
\mathbf{V}^k_h(\Omega_h) \coloneqq 
[V_h^k(\Omega_h)]^3 \, ,
\]
where the above scalar space has been introduced in Section \ref{sec:NC-space}.
Before defining the discrete forms, we observe that the forms introduced in the previous section can be decomposed in local contributions
\[
a(\mathbf{u}, \mathbf{v})
\eqqcolon \sum_{E \in \Omega_h}
a^E(\mathbf{u}, \mathbf{v}) \, ,
\qquad
F(\mathbf{v})
\eqqcolon \sum_{E \in \Omega_h}
F^E(\mathbf{v}) \, .
\]
As is standard in the VEM framework, an explicit expression for the virtual functions is unavailable. Consequently, these functions are typically replaced by a suitable polynomial projection. Since the projection maps onto a lower-dimensional space, the resulting matrix is rank-deficient. To restore the correct rank, a stabilization bilinear form is introduced.
For each element $E$ we define the local bilinear form $a_h^E(\cdot,\cdot)$ as
\begin{equation}\label{eq:ahe}
a_h^E(\mathbf{u}_h, \mathbf{v}_h)
\coloneqq
\int_E
\bigl(
2\mu \mathbf{\Pi}^{0,E}_{k-1}\boldsymbol{\varepsilon}(\mathbf{u}_h) + \lambda \text{tr}(\mathbf{\Pi}^{0,E}_{k-1}\boldsymbol{\varepsilon}(\mathbf{u}_h)) \mathbf{I}
\bigr)
: \mathbf{\Pi}^{0,E}_{k-1}\boldsymbol{\varepsilon}(\mathbf{v}_h) \, {\rm d }E
+
S_h^E\bigl(
(I - \PiZeroE) \mathbf{u}_h, \,
(I - \PiZeroE) \mathbf{v}_h
\bigr) ,
\end{equation}
where the above symbols $\mathbf{\Pi}^{0,E}_{k-1}$ and $\PiZeroE$ denote standard $L^2(E)$ projections on $[\mathbb{P}_{k-1}(E)]^{3 \times 3}$ and $[\mathbb{P}_{k}(E)]^{3}$, respectively.
Furthermore, above we adopt the standard stabilization 
\begin{equation}\label{eq:dofidofi}
S_h^E\bigl(\mathbf{v}_h , \, \mathbf{w}_h \bigr) := \mu \, 
h_E\sum_{i=1}^{\#\text{dof}} \text{dof}^{\,i}_E( \mathbf{v}_h) \, \text{dof}^{\,i}_E( \mathbf{w}_h) \qquad
\forall \mathbf{v}_h, \mathbf{w}_h \in \mathbf{V}^k_h(E) ,
\end{equation}
with the sum including all the Degrees of Freedom for the local space $\mathbf{V}^k_h(E) \coloneqq [V_h^k(E)]^3$, see Section 
\ref{sec:NC-space}.

The local linear form is discretized as
\begin{equation}\label{eq:FhE}
F_h^E(\mathbf{v}_h) = \int_{\Omega} \mathbf{f} \cdot \PiZeroE\mathbf{v}_h \, {\rm d} \Omega + \int_{\Gamma_N} \mathbf{h} \cdot \PiZeroE\mathbf{v}_h \, {\rm d} \Gamma \, .
\end{equation}
The global forms are obtained by summing the local contributions
\[
a_h(\mathbf{u}_h, \mathbf{v}_h)
\coloneqq \sum_{E \in \Omega_h}
a_h^E(\mathbf{u}_h, \mathbf{v}_h) \, ,
\qquad
F_h(\mathbf{v}_h)
\coloneqq \sum_{E \in \Omega_h}
F_h^E(\mathbf{u}_h) \, .
\]
We define the discrete problem as
\begin{equation}\label{eq:prob-disc}
\left\{
\begin{aligned}
&\text{find $\mathbf{u}_h \in \mathbf{V}^k_h(\Omega_h)$ such that:} \\
&a_h(\mathbf{u}_h, \mathbf{v}_h) = F_h(\mathbf{v}_h)\, , \quad \forall \mathbf{v}_h  \in \mathbf{V}^k_h(\Omega_h) \, .
\end{aligned}
\right.
\end{equation}

\begin{remark}
The proposed C-NC framework can be straightforwardly extended to more general constitutive models within the small deformation regime.
Therefore, general linear or nonlinear material laws can be handled by combining the C-NC space construction with the specific VEM discretization techniques proposed, for instance, in \cite{BLM2015}.
\end{remark}

\subsection{Theoretical results}

In this section, we prove stability and interpolation estimates for the proposed method; once these two properties are available, deriving convergence error estimates follows from standard arguments in the literature in the spirit of \cite{BeiraodaVeiga-Brezzi-Marini:2013}.

We start by showing the equivalence between the $H^1$-seminorm and the \texttt{dofi-dofi} stabilization, see definition \eqref{eq:dofidofi}, for virtual functions. The estimate that we want to prove for all $E \in \Omega_h$ is 
$$
h_E \sum_{i=1}^{\#\text{dof}} \text{dof}^{\,i}_E( {\bf v}_h) \, \text{dof}^{\,i}_E( {\bf v}_h)
\lesssim
| {\bf v}_h |_{1,E}
\lesssim
h_E \sum_{i=1}^{\#\text{dof}} \text{dof}^{\,i}_E( {\bf v}_h) \, \text{dof}^{\,i}_E( {\bf v}_h) \, .
$$
for all ${\bf v}_h \in {\bf V}^k_h(E)$, possibly under additional suitable conditions, with hidden constants independent of the particular element $E$ in the shape-regular family.

We focus on the scalar case, as the DoFs in $\mathbf{V}_h^k(\Omega_h)$ are simply the vector-valued counterparts of those for $V_h^k(\Omega_h)$.
Since the first bound is fairly simple, we show the proof briefly.
\begin{proposition}\label{prp:beirao-inequality}
Let the mesh assumptions in Section \ref{sec:mesh-ass} hold. 
Then, for all $E \in \Omega_h$ and $v_h \in \ker(\PiNablaE)$ we have
\begin{equation}\label{eq:upper}
h_E \sum_{i=1}^{\#\text{dof}} | \mathrm{dof}^i_E (v_h)|^2 \lesssim | v_h |_{1,E}^2 \, .
\end{equation}
\end{proposition}
\begin{proof}
There are three type of DoFs appearing in the left hand side sum, see the DoF list for $V^k_h(E)$ in Section \ref{sec:NC-space}. Volume moment DoFs of $v_h$ are easily bounded applying a H\"older inequality (recall that $\| \varphi_i\|_{L^\infty(E)}=1$), the scaling $|E| \simeq h_E^3$ and a scaled Poincar\'e inequality using $v_h \in \ker(\PiNablaE)$:
$$
\frac{1}{|E|} \int_E v_h \, \varphi_i \lesssim  \frac{1}{|E|^{1/2}} \| v_h \|_{0,E}
\lesssim h_E^{-1/2} |v_h|_{1,E} \, ,
$$
which shows that such DoFs satisfy \eqref{eq:upper}.
Face moment DoFs are handled analogously, also using a scaled trace inequality that allows to bound the $L^2(F)$ norm with a scaled $H^1(E)$ norm. Therefore, using again the above scaled Poincar\'e inequality, one obtains
\begin{equation}\label{koalapower}
\frac{1}{|F|} \int_F v_h \, \varphi_i \lesssim \frac{1}{|F|^{1/2}} \| v_h \|_{0,F}
\lesssim \frac{h_F^{-1/2}}{|F|^{1/2}} \| v_h \|_{0,E} + \frac{h_F^{1/2}}{|F|^{1/2}} | v_h |_{1,E} 
\lesssim h_E^{-1/2} |v_h|_{1,E} \, .
\end{equation}
Regarding the pointwise value DoFs, these are always associated to the evaluation of $v_h$ on a node ${\bf x}$ that lays on the boundary of a face $F \in {\cal F}^c_E$. Since on such faces the function $v_h$ corresponds by definition to a standard two-dimensional $H^1$-conforming VEM we can apply Lemma 4.2 in \cite{Chen-Huang:2018}, yielding 
$$
|v_h({\bf x})| \lesssim h_F^{-1} \| v_h \|_{0,F}
$$
so that also this case can be concluded as in \eqref{koalapower}.
\end{proof}

The second estimate is more involved. We start by proving some preliminary lemmas. 
\begin{lemma}\label{lem:w:prop}
Let $v_h \in V^k_h(E)$ and $w \in H^1(E)$ be such that:
\begin{itemize}
    \item $w = v_h$ on $F$, for all $F \in \mathcal{F}^c_E$;
    \item $\Pi^{0,F}_{k-1}(w - v_h) = 0$ for all $F \in \mathcal{F}^{nc}_E$;
    \item $\PiZeroE(w -v_h) = 0$.
\end{itemize}
Then, the following estimate holds:
\[
| v_h |_{1,E} \leq | w |_{1,E} \, .
\]
\end{lemma}

\begin{proof}
By applying integration by parts, we have
\[
\int_E \nabla v_h \cdot \nabla (w - v_h) \, {\rm d} E = 
-\int_E (w - v_h) \Delta v_h \, {\rm d} E + \sum_{F \in \mathcal{F}^c_E} \int_F (w - v_h) \partial_\mathbf{n} v_h \, {\rm d}F + \sum_{F \in \mathcal{F}^{nc}_E} \int_F (w - v_h) \partial_\mathbf{n} v_h \, {\rm d}F.
\]
The three terms in the right hand side are all equal to zero.
Indeed, the first integral vanishes because $\Delta v_h \in \mathbb{P}_{k}(E)$ and, by hypothesis $\Pi_{k}^{0,E}(v_h - w) = 0$.
On the conforming faces $F \in \mathcal{F}^c_E$, the functions have the same trace, thus also the second integral is identically zero. On the nonconforming faces $F \in \mathcal{F}^{nc}_E$, we exploit the fact that the normal derivative $\partial_\mathbf{n} v_h$ is a polynomial of degree $k-1$ on $F$ combined with the second item in the above hypotheses.
Consequently, since the above integral vanishews, the following orthogonal decomposition holds:
\[
| w |^2_{1,E} = | w - v_h |^2_{1,E} + | v_h |^2_{1,E},
\]
yielding the desired estimate $| v_h |_{1,E} \leq | w |_{1,E}$.
\end{proof}

\begin{lemma}\label{lm:PiNabla_stability}
Let the mesh assumptions in Section \ref{sec:mesh-ass} hold. 
Then, the following estimate holds for all $E \in \Omega_h$ and $v_h \in V_h^k(E)$:
\begin{equation}\label{eq:H1_estimate}
| \PiNablaE v_h |_{1,E} + h_E^{-1} \| \PiNablaE v_h \|_{0,E}  
\lesssim
h_E^{-1} \| \Pi^{0,E}_{k-2} v_h \|_{0,E}
+ \sum_{F \in \mathcal{F}^{c}_E} h_E^{-1/2} \| v_h \|_{0,F}
+ \sum_{F \in \mathcal{F}^{nc}_E} h_E^{-1/2} \| \Pi^{0,F}_{k-1} v_h \|_{0,F} \, .
\end{equation}
\end{lemma}

\begin{proof}
By exploiting integration by parts, we obtain:
\[
\begin{aligned}
\int_E |\nabla \PiNablaE v_h|^2 \, {\rm d} E
&= \int_E \nabla \PiNablaE v_h \cdot \nabla v_h \, {\rm d} E \\
&= -\int_E v_h \Delta \PiNablaE v_h \, {\rm d} E \\
& \quad + \sum_{F \in \mathcal{F}^c_E} \int_F v_h \bigl(\partial_\mathbf{n} \PiNablaE v_h \bigr) \, {\rm d}F
+ \sum_{F \in \mathcal{F}^{nc}_E} \int_F v_h \bigl(\partial_\mathbf{n} \PiNablaE v_h \bigr) \, {\rm d}F \\
&= -\int_E \Pi^{0,E}_{k-2} v_h \Delta \PiNablaE v_h \, {\rm d} E \\
& \quad + \sum_{F \in \mathcal{F}^c_E} \int_F v_h \bigl(\partial_\mathbf{n} \PiNablaE v_h \bigr) \, {\rm d}F
+ \sum_{F \in \mathcal{F}^{nc}_E} \int_F \Pi^{0,F}_{k-1} v_h \bigl(\partial_\mathbf{n} \PiNablaE v_h \bigr) \, {\rm d}F \, .
\end{aligned}
\]
Using polynomial inverse inequalities, trace inequalities, and the Cauchy-Schwarz inequality, we deduce:
\begin{equation}\label{eq:PiNablaEst1}
| \PiNablaE v_h |^2_{1,E}
\lesssim
\left(
h_E^{-1} \| \Pi^{0,E}_{k-2} v_h \|_{0,E}
+ \sum_{F \in \mathcal{F}^{c}_E} h_E^{-1/2} \| v_h \|_{0,F}
+ \sum_{F \in \mathcal{F}^{nc}_E} h_E^{-1/2} \| \Pi^{0,F}_{k-1} v_h \|_{0,F}
\right)
| \PiNablaE v_h |_{1,E}.
\end{equation}
We also observe that:
\begin{equation}\label{eq:PiNablaEst-2}
\left| \int_{\partial E} \PiNablaE v_h \, {\rm d}s \right|
= \left| \sum_{F \subset \partial E} \int_{F} \Pi^{0,F}_0 v_h \, {\rm d}F \right|
\lesssim \sum_{F \in \mathcal{F}^{c}_E} h_E^{1/2} \| v_h \|_{0,F}
+ \sum_{F \in \mathcal{F}^{nc}_E} h_E^{1/2} \| \Pi^{0,F}_{k-1} v_h \|_{0,F}.
\end{equation}
Combining \eqref{eq:PiNablaEst1} and \eqref{eq:PiNablaEst-2} a scaled Poincar\'e-Friedrichs inequality, we easily obtain the desired estimate.
\end{proof}
Arguing as Lemma 2.18 in \cite{Brenner2017}, we immediately derive the following corollary
\begin{corollary}\label{cor:easy}
Let the mesh assumptions in Section \ref{sec:mesh-ass} hold. 
Then, the following estimate holds for all $E \in \Omega_h$ and $v_h \in V_h^k(E)$:
\begin{equation}
\| \PiZeroE v_h \|_{0,E}
\lesssim
\| \Pi^{0,E}_{k-2} v_h \|_{0,E}
+ \sum_{F \in \mathcal{F}^{c}_E} h_E^{1/2} \| v_h \|_{0,F}
+ \sum_{F \in \mathcal{F}^{nc}_E} h_E^{1/2} \| \Pi^{0,F}_{k-1} v_h \|_{0,F} \, .
\end{equation}
\end{corollary}

\begin{lemma}\label{lm:beiraos-lemma}
Let the mesh assumptions in Section \ref{sec:mesh-ass} hold.
Let $v_h \in V_h^k(E)$, $E \in \Omega_h$. Then it exists a function $w \in H^1(E)$ that satisfies the first and second conditions in Lemma \ref{lem:w:prop}, plus the additional property
\begin{equation}\label{eq:equiv:br}
h_E |w|_{1,E} \simeq \| w \|_{0,E} \simeq \sum_{F \in \mathcal{F}^{c}_E} h_E^{1/2} \| v_h \|_{0,F}
+ \sum_{F \in \mathcal{F}^{nc}_E} h_E^{1/2} \| \Pi^{0,F}_{k-1} v_h \|_{0,F}  \, .
\end{equation}
\end{lemma}
\begin{proof}
We start by defining $w$ on $\partial E$, which is chosen in the following set:
\begin{equation}\label{eq:boundaryspace}
w|_{\partial E} \in 
\big\{ v \in C^0(\partial E) \ : \ v|_F \in V_h^{k+1}(F) \ \forall \textrm{ face } F \in \partial E \big\} \, ,
\end{equation}
see Section \ref{sec:facevem}. We then select the specific $w$ by setting $w|_F = v_h |_F$ for all $F \in {\cal F}^c_E$, and $\Pi^{0,F}_{k-1} (w|_F) = \Pi^{0,F}_{k-1} (v_h |_F)$ for all $F \in {\cal F}^{nc}_E$. Note that the latter is an admissible condition due to the degrees of freedom of $V_h^{k+1}(F)$, c.f. Section \ref{sec:facevem}. Finally, the remaining DoFs associated to the above boundary space \eqref{eq:boundaryspace} (i.e. the nodal evaluations on edges that do not lay on $\partial F$ for some $F \in {\cal F}^c_E$) are set to zero. 

In order to lift $w|_F$ to a function living on the whole $E$, we follow the same identical procedure proposed in Lemma 5.3 of \cite{Brenner2017}. This immediately yields, with the same identical proof as in the aforementioned lemma, the norm equivalences in \eqref{eq:equiv:br}.
Since $w$ also satisfies, by construction, the first two conditions in Lemma \ref{lem:w:prop}, the proof is concluded.
\end{proof}

\begin{proposition}\label{prop:judoboy}
Let the mesh assumptions in Section \ref{sec:mesh-ass} hold.
Then  the following estimate holds for all $E \in \Omega_h$ and $v_h \in V_h^k(E)$
\[
|v_h|_{1,E} 
\lesssim h_E^{-1} \| \Pi^{0,E}_{k-2} v_h \|_{0,E}
+ \sum_{F \in \mathcal{F}^{c}_E} h_E^{-1/2} \| v_h \|_{0,F}
+ \sum_{F \in \mathcal{F}^{nc}_E} h_E^{-1/2} \| \Pi^{0,F}_{k-1} v_h \|_{0,F}
\]
\end{proposition}
\begin{proof}
Given $B$, the ball contained in $E$ associated to the shape-regularity condition in Section \ref{sec:mesh-ass},
we denote by $\phi$ a bubble function with support in $B \subset E$ such that $\| \phi \|_{L^\infty(E)} = 1$.
Then, we select a polynomial $p \in \mathbb{P}_k(E)$ such that the function defined as $z = w + \phi p$ satisfies
\[
\int_E (v_h - z) q \, {\rm d }E = 0 \, , \qquad \forall q \in \mathbb{P}_k(E) \, ,
\]
where $w$ is the function introduced in Lemma \ref{lm:beiraos-lemma}.
We note that, since the function $z$ satisfies the assumptions in Lemma \ref{lem:w:prop} by construction, it holds
\[
|v_h|_{1,E} \leq | z |_{1,E} \, .
\]
By standard polynomial properties and by definition of $z$, we have that
\[
\| p \|^2_{0,E}
\lesssim
\int_E \phi p^2 {\rm d}E
=
\int_E (\PiZeroE v_h - w) \, p \,  {\rm d}E \, ;
\]
therefore the Cauchy-Schwarz inequality gives
\[
\| p \|_{0,E} 
\leq 
\| \PiZeroE v_h\|_{0,E} + \| w \|_{0,E} \, .
\]
Exploiting Corollary \ref{cor:easy} and Lemma \ref{lm:beiraos-lemma}, we deduce
\[
\| p \|_{0,E} 
\lesssim
 \| \Pi^{0,E}_{k-2} v_h \|_{0,E}
+ \sum_{F \in \mathcal{F}^{c}_E} h_E^{1/2} \| v_h \|_{0,F}
+ \sum_{F \in \mathcal{F}^{nc}_E} h_E^{1/2} \| \Pi^{0,F}_{k-1} v_h \|_{0,F} \, .
\]
In conclusion, we obtain the desired estimate 
combining the above bounds with some manipulations and an inverse estimate for polynomials
\[
\begin{aligned}
|v_h|_{1,E} &\leq | z |_{1,E}
\leq | w |_{1,E} + | \phi p|_{1,E} \\
&\leq h_E^{-1} \| \Pi^{0,E}_{k-2} v_h \|_{0,E}
+ \sum_{F \in \mathcal{F}^{c}_E} h_E^{-1/2} \| v_h \|_{0,F}
+ \sum_{F \in \mathcal{F}^{nc}_E} h_E^{-1/2} \| \Pi^{0,F}_{k-1} v_h \|_{0,F} \, .
\end{aligned}
\]
 \end{proof}

\begin{corollary} \label{cor:stab_upper_bound}
Let the mesh assumptions in Section \ref{sec:mesh-ass} hold.
Then  the following estimate holds for all $E \in \Omega_h$ and $v_h \in V_h^k(E)$
\begin{equation}
|v_h|_{1,E}^2 \lesssim h_E \sum_{i=1}^{\#dof} |\text{dof}_E^i(v_h)|^2.
\end{equation}
\end{corollary}

\begin{proof}
We proceed by bounding each term that appears in the right-hand side of Proposition \ref{prop:judoboy}.
The projection $\Pi^{0,E}_{k-2} v_h$ is determined by the interior DoFs in the element $E$. Since the Dofs scale independently of the element size \cite{BeiradaVeiga-Brezzi-Marini-Russo:2014}, a standard scaling argument gives
$$\| \Pi^{0,E}_{k-2} v_h \|_{0,E}^2 \simeq |E| \sum_{i \in \mathcal{N}^{E}_{int}} |\text{dof}_E^i(v_h)|^2 \simeq h_E^3 \sum_{i \in \mathcal{N}^{E}_{int}} |\text{dof}_E^i(v_h)|^2 \, ,$$
where $\mathcal{N}^{E}_{int}$ denotes the index set of the internal DoFs of $E$.
Similarly, on the nonconforming faces, it holds that
$$\| \Pi^{0,F}_{k-1} v_h \|_{0,F}^2 \simeq |F| \sum_{i \in \mathcal{N}^{F}_{nc}} |\text{dof}_F^{\,i}(v_h)|^2 \simeq h_E^2 \sum_{i \in \mathcal{N}^{F}_{nc}} |\text{dof}_F^{\,i}(v_h)|^2 \, .$$

Finally, for the conforming faces $F \in \mathcal{F}^{c}_E$, the restriction $v_h|_F$ is a 2D Virtual Element function. By applying standard stabilization results for 2D VEM functions \cite{BeiraodaVeiga-Lovadina-Russo:2017,Chen-Huang:2018}, we thus obtain
$$
\| v_h \|_{0,F}^2 \simeq |F| \sum_{i \in \mathcal{N}^{F}_{c}} |\text{dof}_F^{\,i}(v_h)|^2 \simeq h_E^2 \sum_{i \in \mathcal{N}^{F}_{c}} |\text{dof}_F^{\,i}(v_h)|^2 \, .
$$
Substituting these three bounds in Proposition \ref{prop:judoboy}, we obtain the desired estimate.
\end{proof}

We conclude this section by showing the following interpolation estimate.
\begin{proposition}[Interpolation estimates]\label{lm:portugal-interpolation}
Let the mesh assumptions in Section \ref{sec:mesh-ass} hold.
Let the integer $m$ satisfy $2 \le m \le k+1$ and $u \in H^{m}(\Omega_h) \cap H_0^1(\Omega)$. 
Then there exists $u_I \in V_h^k(\Omega_h)$ such that
\[
| u - u_I |_{1,E} \lesssim h_E^{m-1} | u |_{m,E} \quad \forall E \in \Omega_h \, .
\]
\end{proposition}
\begin{proof}
We fix an element $E \in \Omega_h$. Let $p_k \in \Pol_k(E)$ be a polynomial approximant of $u$. Using the triangle inequality, it holds
\begin{equation}\label{eq:1}
| u - u_I |_{1,E} \leq | u - p_k |_{1,E} + | p_k - u_I |_{1,E}.
\end{equation}
The first term is controlled by standard approximation estimates for polynomials on star shaped domains (see for instance \cite{dupont-scott}):
\begin{equation}\label{eq:2}
| u - p_k |_{1,E} \lesssim  h_E^{m-1} | u |_{m,E}.
\end{equation}
For the second term, since the interpolation operator preserves polynomials of degree $k$, we have $p_k = (p_k)_I$, which implies $p_k - u_I = (p_k - u)_I$. Hence using Proposition \ref{prop:judoboy}, we can bound  the norm of $(p_k - u)_I$ as
\[
\begin{aligned}
|(p_k - u)_I|_{1,E}
&\lesssim
h_E \sum_{\text{dof}} \text{dof}^{\,i}_E((p_k - u)_I \bigr) \text{dof}^{\,i}_E\bigl((p_k - u)_I \bigr) 
= h_E \sum_{\text{dof}} | \text{dof}^{\,i}_E( p_k - u ) |^2
\, . 
\end{aligned}
\]
Using standard arguments and the definition of the degrees of freedom, it is easy to check that 
$$
\begin{aligned}
h_E \sum_{\text{dof}} | \text{dof}^{\,i}_E(p_k - u) |^2 
& \lesssim h_E^{-2} \| p_k - u \|_{0,E}^2 + h_E^{-1} \sum_{F \in \mathcal{F}_E} \| p_k - u \|_{0,F}^2
+ h_E \| p_k - u \|_{L^\infty(E)}^2 \\
& \lesssim h_E \| p_k - u \|_{L^\infty(E)}^2 \, ,
\end{aligned}
$$
which in turn is bounded by standard polynomial approximation results on star-shaped domains in the $L^\infty$ norm. 
Combining the two inequalities above we obtain  
\begin{equation}\label{eq:3}
|(p_k - u)_I|_{1,E} \lesssim h_E^{m-1} | u |_{m,E}.
\end{equation}
The result now follows from \eqref{eq:1}, \eqref{eq:2} and \eqref{eq:3}.
\end{proof}

Finally, we can state the following convergence result.

\begin{theorem}\label{th:main}
Let the mesh assumptions in Section \ref{sec:mesh-ass} hold.
The discrete problem \eqref{eq:prob-disc} has a unique solution $\mathbf{u}_h \in \mathbf{V}_h^k(\Omega_h)$. Moreover, if the solution $\mathbf{u}$ of the continuous problem satisfies $\mathbf{u} \in H^m(\Omega_h)$ with $1 \leq m \leq k+1$, the following estimate holds:
\[
| \mathbf{u} - \mathbf{u}_h |_{1,\Omega_h} 
\lesssim 
\sum_{E \in \Omega_h} h_E^{m-1} | \mathbf{u} |_{m,E} \, .
\]
\end{theorem}
\begin{proof}
By virtue of Proposition \ref{prp:beirao-inequality} and Corollary \ref{cor:stab_upper_bound}, it is trivial to check by standard VEM arguments \cite{BeiradaVeiga-Brezzi-Cangiani-Manzini-Marini-Russo:2013} that, for every $E \in \Omega_h$, we have
\begin{equation}\label{eq:equiv}
a^E(\mathbf{v}_h, \mathbf{v}_h) \lesssim a_h^E(\mathbf{v}_h, \mathbf{v}_h) \lesssim a^E(\mathbf{v}_h, \mathbf{v}_h) \quad \forall \mathbf{v}_h \in \mathbf{V}_h^k(E) \, .
\end{equation}
The results now follows by the same argument in Theorem 3.1 of \cite{BeiradaVeiga-Brezzi-Cangiani-Manzini-Marini-Russo:2013},
using the above equivalence bound \eqref{eq:equiv}, the interpolation estimate in Proposition \ref{lm:portugal-interpolation} and the definition of the right-hand side \eqref{eq:FhE}.
Finally, the nonconformity error of the space can be treated as in \cite{AyusodeDios-Lipnikov-Manzini:2016, curvo-Yi}.
\end{proof}

We note that the hidden constant in the result above depends on the incompressibility constant $\lambda$. Following the same techniques as in \cite{BeiraodaVeiga-Brezzi-Marini:2013}, one could also show that for $k \ge 2$ the proposed scheme is robust in the quasi-incompressible limit, in the following sense.

\begin{corollary}
Let the same assumptions and notations as in Theorem \ref{th:main}, with the additional condition $k \ge 2$. Then 
\[
| \mathbf{u} - \mathbf{u}_h |_{1,\Omega_h} 
\lesssim 
\sum_{E \in \Omega_h} h_E^{m-1} ( | \mathbf{u} |_{m,E} + | {p} |_{m-1,E} )\, ,
\]
where the hidden constant is independent of $\lambda$ and $p = \lambda \, {\rm div} u$.
\end{corollary}

Note that requiring boundedness of $\lambda \, {\rm div} u$, the volumetric stress intensity, is an acceptable assumption.

\section{Application to the curved case of curved domains \texorpdfstring{$\Omega$}{Omega}}\label{sec:3}

In this section, we investigate a possible application of the C-NC virtual element space.
We consider the same linear elastic problem considered in the previous section, but now assume that the domain $\Omega$ has a piecewise regular but possibly curved boundary. We now observe the following.
\begin{enumerate}
\item It is well known that using polyhedral (tets and hexahedra included) elements, that is with straight edges, leads to a geometry approximation that, for elements of second or higher order, typically dominates the error.
\item While many conforming VEM with curved faces exist in 2D \cite{BeiraodaVeiga-Russo-Vacca:2019, Bertoluzza-Pennacchio-Prada:2019, Aldakheel2020, BeiraodaVeiga-Brezzi-Marini-Russo:2020}, developing VEM with curved faces in 3D is complicated \cite{curved3D} due to the need to define the VEM space on curved surfaces, combined with the need of computing projections. 
\item On the contrary, curved non-conforming VEM in 3D have been developed in \cite{curvo-Yi} but, from the engineering point of view, when applied in solid mechanics suffer from the drawback of being based on a piecewise discontinuous displacement $u_h$, which is unnatural for a continuous body. 
\end{enumerate}

Our idea in the present section is to apply our conforming-nonconforming elements to the present case, exploiting non-conforming faces to describe the curved boundary and using internal (straight) conforming faces internally to $\Omega$, thus guaranteeing the displacement continuity in the elastic body.

\subsection{Mesh assumptions and discrete space}

Let $\Omega_h$ denote a tessellation of the domain $\Omega$ with possibly curved faces on the boundary. 
We assume that the internal elements of $\Omega_h$ are standard polyhedra with planar faces, whereas the elements adjacent to the boundary may feature curved faces or edges lying exactly on $\partial \Omega$.
For every boundary face $F$, we assume the existence of a sufficiently smooth parametrization $\gamma_F: \widehat{F} \to F$, where $\widehat{F}$ is a planar reference polygon (such parametrization could, for example, be inherited by a CAD parametrization of the domain boundary).
We introduce the following mapped polynomial space for curved faces ($m \in {\mathbb N}$):
\[
\widetilde{\mathbb P}_m(F) 
= 
\{
\widetilde q = q \, \circ \, \gamma_F^{-1} \, \text{ s.t. }\, q \in \mathbb P_m(\widehat{F})
\} \, .
\]
Clearly, if $F$ is not a curved face, this space coincides with the standard polynomial space.
We introduce the C-NC virtual space in the curved case
\begin{equation}\label{pokeball2}
\begin{aligned}
V_h^k(E) := \Big\{ v_h \in H^1(E)\, & :\, \Delta v_h \in \Pol_{k}(E) \, , \, 
v_h|_{\Gamma^{c}_E} \in C^0(\Gamma^{c}_E) \, , \\
& \int_E (v_h-\PiNablaE v_h) p_k \,{\rm d}E = 0 \ \forall \ p_k \in \Pol_k/\Pol_{k-2}(E) \, , \\
&v_h|_F \in V_h^k(F) \ \forall \ F \in {\cal F}^c_E \, , \,
(\partial_{{\bf n}} v_h)|_F \in \widetilde{\Pol}_{k-1}(F) \ \forall \ F \in {\cal F}^n_E  \Big\} .
\end{aligned}
\end{equation}
In this setting, the set ${\cal F}^n_E$ contains all the boundary faces, which may be potentially curved, while the set ${\cal F}^c_E$ contains all the internal (straight) faces.

\begin{remark}
 It is important to observe that the set ${\cal F}^c_E$ of conforming faces may include planar faces with one or more curved edges. This occurs specifically when an internal face shares an edge with a boundary face $F \in {\cal F}^n_E$. While the face surface itself remains planar, its boundary $\partial F$ inherits the curvature from the domain boundary $\partial \Omega$. In this scenario, the construction of the virtual space on $F$ follows the framework developed for 2D curved VEM \cite{BeiraodaVeiga-Russo-Vacca:2019}
\begin{equation}
\begin{aligned}
V_h^k(F) := \Big\{ v \in H^1(F) \cap C^0(\partial F) \text{ s.t. } & \Delta v \in \mathbb{P}_{k}(F) \, , \\
& v|_e \in \begin{cases} \widetilde{\mathbb{P}}_k(e) & \text{if } e \subset \partial \Omega \\ \mathbb{P}_k(e) & \text{otherwise} \end{cases} \qquad \forall e \subset \partial F \, , \\
& \int_F (v - \Pi^{\nabla, F}_k v) p_k \, {\rm d}F = 0 \ \forall p_k \in \mathbb{P}_k / \mathbb{P}_{k-2}(F) \Big\} ,
\end{aligned}
\end{equation}
where $\widetilde{\mathbb{P}}_k(e)$ denotes the space of mapped polynomials on the curved edge $e$ (defined analogously to $\widetilde{\mathbb{P}}_m(F)$).
\end{remark}

The degrees of freedom are defined as follows:
\begin{itemize}
\item pointwise evaluation $v_h(\nu)$ at all vertexes $\nu$ of $E$ which lay in $\Gamma^c_E$;
\vspace{-2mm}
\item pointwise evaluation of $v_h$ at $k-2$ distinct points for each edge $e$ of $E$ which lays in $\Gamma^c_E$;
\vspace{-2mm}
\item moments $\frac{1}{|F|} \int_F v_h \, \varphi_i$, $i=1,..,M$ with $\{ \varphi_i\}_{i=1}^M$ basis of $\Pol_{k-2}(F)$ for all $F \in {\cal F}^c_E$ ;
\vspace{-2mm}
\item moments $\frac{1}{|F|} \int_{F} v_h  \widetilde \varphi_i$, $i=1,\dots,K$, with $\{ \widetilde{\varphi}_i\}_{i=1}^K$ being a basis of $\widetilde\Pol_{k-1}(F)$, for all $F \in {\cal F}^n_E$
\vspace{-2mm}
\item moments $\frac{1}{|E|} \int_E v_h \, \varphi_i$, $i=1,..,L$ with $\{ \varphi_i\}_{i=1}^L$ basis of $\Pol_{k-2}(E)$.
\end{itemize}

Once the conforming/nonconforming faces have been identified accordingly to the above rule and the degrees of freedom are available, the construction of the method follows the same identical procedure as for the straight edge case in Section \ref{sec:elast:discr}. The main differences is that many integrals (of known functions and polynomials) will now be on curved polyhedra and curved faces, which requires advanced integration rules (see for instance \cite{chinsuku,sommariva2009gauss}).

Concerning the computation of the energy projection $\PiNablaE$, by the same calculation as in \eqref{eq:squirtle} its right-hand side requires evaluating the following boundary terms on the nonconforming faces
\[
\sum_{F \in {\cal F}^n_E}
\int_F (\partial_\mathbf{n} p_k) v_h \, {\rm d}s \, ,
\]
which, following the idea in \cite{curvo-Yi}, we approximate with 
\begin{equation}\label{eq:zzz}
\sum_{F \in {\cal F}^n_E}
\int_F (\partial_\mathbf{n} p_k) \widetilde \Pi^{0,F}_{k-1} v_h \, {\rm d}s \, ;
\end{equation}
see Sections 3.2 (and also Section 7) of \cite{curvo-Yi} for the definition of such projector.

Note that such step is indeed an approximation since, differently from the planar face case, the scalar $\partial_\mathbf{n} p_k$ on $F$ is not in $\widetilde{\mathbb P}_{k-1}(F)$.
It is immediate to check that the quantity in \eqref{eq:zzz} is directly and exactly computable using solely the nonconforming degrees of freedom associated with the face $F$. By plugging this projection into the integral, we successfully assemble the right-hand side. The same identical observation/approximation is considered in the calculation of the projection operator in \eqref{eq:pizerograd}.

\section{Numerical tests} \label{sec:4}

In this section, we present the numerical results for the C--NC method across three different scenarios. 
Specifically, in Sections~\ref{sec:convSt} and~\ref{sec:convCur}, we
numerically verify the expected convergence rates for domains with planar faces and curved faces, respectively. 
Finally, Section~\ref{sec:convGenCur} provides a proof of concept for the method's ability 
to handle curved surfaces tessellated with arbitrary (possibly non planar) polygonal elements.
All numerical experiments have been implemented using the \texttt{C++} library \texttt{vem++}~\cite{vem++}.

The computational domain and the data of the PDEs are specified in each subsection. 
However, we always evaluate the accuracy of the method using the following error indicators
\[
e_{L^2} := \left( \sum_{E\in\Omega_h}\left\|\uu-\Pi^0_k \uu_h\right\|^2_{0,E} \right)^{1/2} \, ,
\qquad
e_{H^1} := \left( \sum_{E\in\Omega_h}\left\|\nabla\uu-\nabla\left(\Pi_k^\nabla \uu_h\right)\right\|^2_{0,E} \right)^{1/2} \, ,
\]
where \(\uu\) is the exact displacement solution. 
Here, \(\Pi^0_k \uu_h\) is the \(L^2\) polynomial vector projection of the discrete solution \(\uu_h\), 
whose components are defined in Equation~\eqref{eqn:scalPiZero}. 
The projection \(\Pi_k^\nabla \uu_h\) is constructed in a similar way 
but employing the scalar \(H^1\) projection defined in Equation~\eqref{eqn:scalPiNabla} for each component.
In view of the theoretical results,
we aim to verify that $e_{L^2}=\mathcal{O}(h^{k+1})$ and $e_{H^1}=\mathcal{O}(h^{k})$.
 
\subsection{Fully planar boundaries}\label{sec:convSt}

In this experiment, we investigate the convergence of the C--NC method in a domain where all boundaries are planar. 
We consider the unit cube $\Omega=(0, 1)^3$ as the computational domain. 
The Lamé coefficients are set to $\lambda = 2$ and $\mu = 1$, 
and the right-hand side of Equation~\eqref{eq:problemacontinuo} is chosen such 
that the exact displacement field is given by:
$$
\uu = \begin{pmatrix}
e^z + \sin(x) + \sin (y) \\
e^y + \sin(x) + \sin (z) \\
e^x + \sin(y) + \sin (z)
\end{pmatrix}.
$$
We impose Dirichlet boundary conditions on all faces except for \(z=0\) and \(z=1\), 
where Neumann boundary conditions are applied. 

To obtain a polyhedral mesh of \(\Omega\),
we generate the initial Voronoi tessellation with \texttt{voro++}~\cite{voro++}
which is then optimized via a Lloyd's algorithm.
We verify the convergence of the error indicators \(e_{L^2}\) and \(e_{H^1}\) 
by computing the errors on a sequence of four meshes with decreasing mesh size \(h\).
In Figure~\ref{fig:cubeVoro}, we show one of these meshes and 
its exploded view to better highlight the elements' shapes.

\begin{figure}[!htb]
\centering
\includegraphics[width=0.8\textwidth]{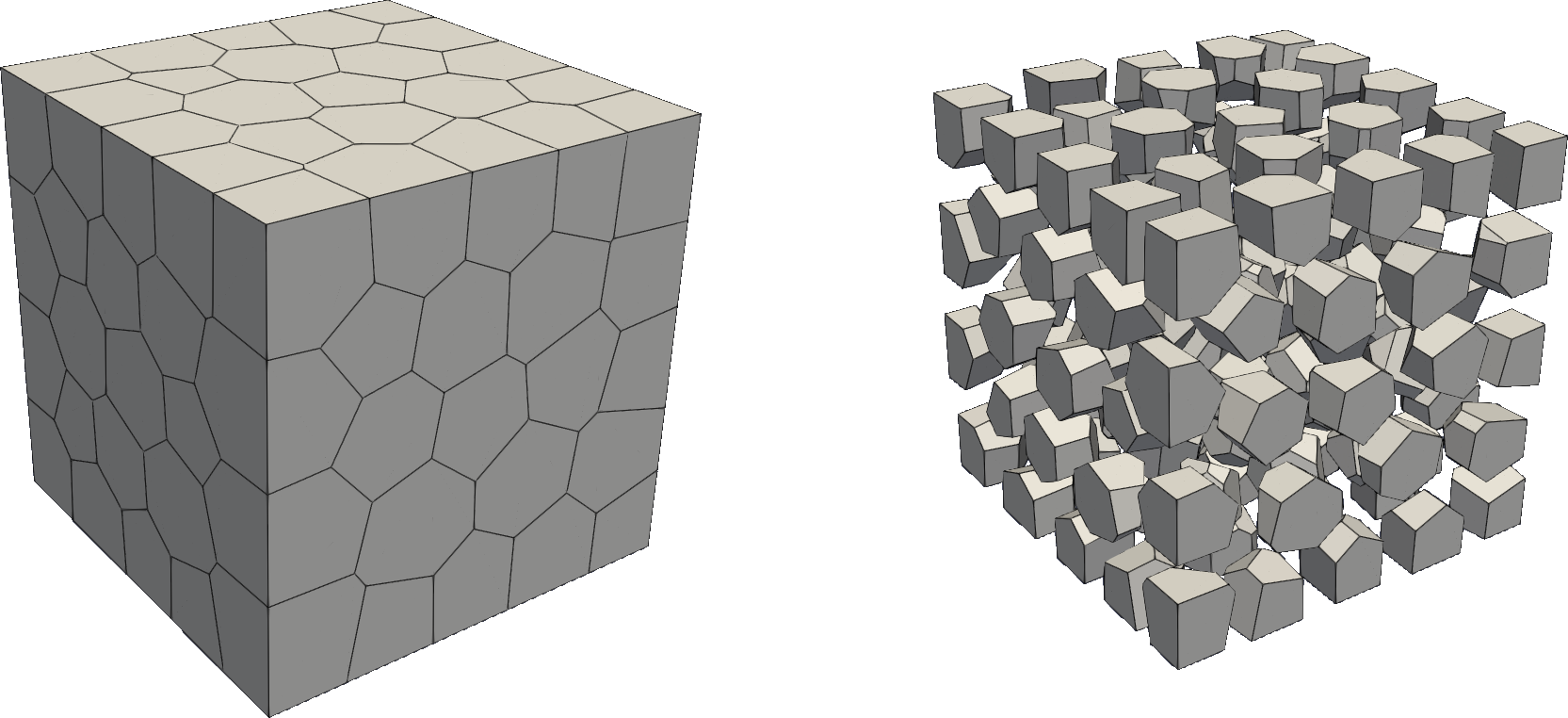}
\caption{The full cube tessellated, on the left, the exploded view to highlight the interior, right.}
\label{fig:cubeVoro}
\end{figure}

The global discrete space \(\mathbf{V}^k_h(\Omega_h)\) is built in such a way that 
all polygons lying on the plane \(z=1\) are non-conforming faces; 
specifically, for each element \(E\) with a face \(F\) on \(z=1\), it holds that \(F\in \mathcal{F}^n_E\). All the remaining faces of the mesh are conforming, i.e. in ${\cal F}^c$.
We consider the VEM approximation degrees $k=1$ and $2$.

The convergence lines are depicted in Figure~\ref{fig:test1-results}.
Both error indicators exhibit the expected convergence rates;
specifically, $e_{L^2}$ behaves as $\mathcal{O}(h^{k+1})$,
while $e_{H^1}$ behaves as $\mathcal{O}(h^{k})$. 

\begin{figure}[htbp]
\centering
\begin{tabular}{cc}
\includegraphics[width=0.42\textwidth]{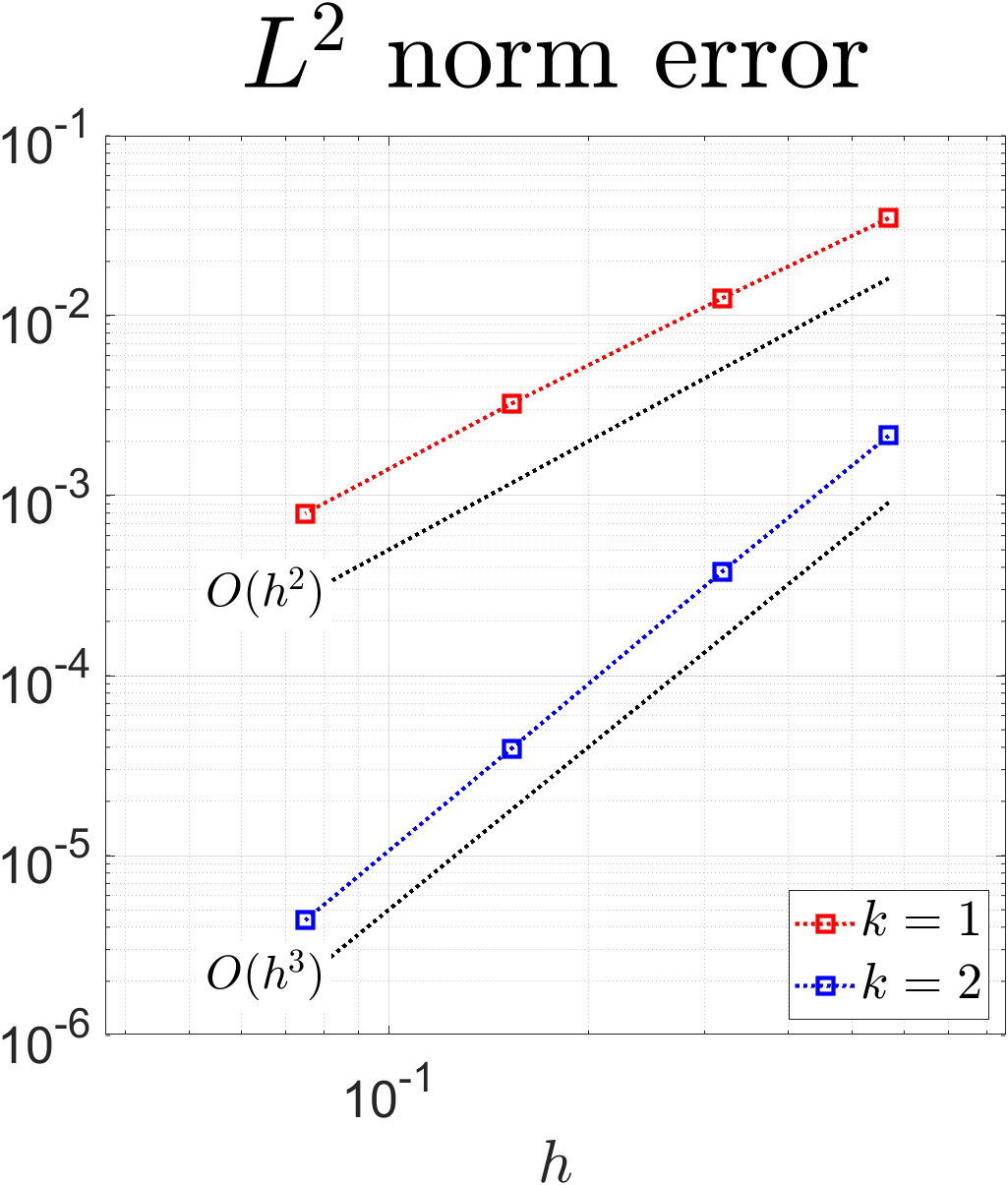}  &
\includegraphics[width=0.42\textwidth]{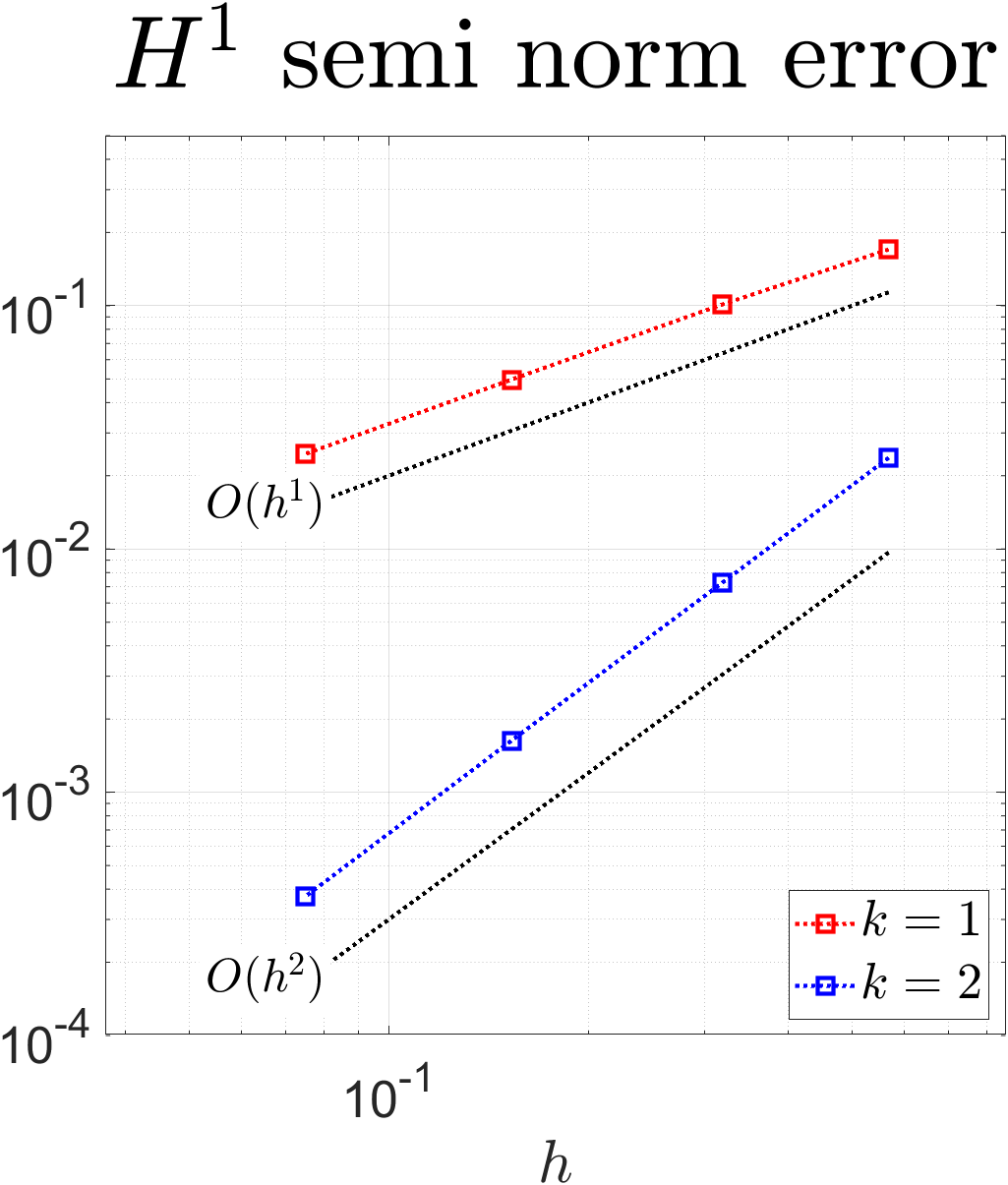}
\end{tabular}
\caption{Fully planar boundaries: convergence lines of $e_{L^2}$ and $e_{H^1}$.}
\label{fig:test1-results}
\end{figure}

\subsection{Planar and curved faces}\label{sec:convCur}

In this experiment, we investigate the convergence of the C--NC method in a domain characterized by both planar and curved boundaries. 
Specifically, we consider a cylinder of radius 1 and height 2 as the computational domain; consequently, the top and bottom bases are planar, while the lateral surface is curved.

The Lamé coefficients are set to $\lambda = \mu = 1$. 
We consider the exact displacement field given by:
\[
\uu(x, y, z) = 
\begin{pmatrix}
e^x \sin(y+z) \\
e^y \sin(x+z) \\
e^z \sin(x+y)
\end{pmatrix},
\]
and the data of the problem are prescribed accordingly. 
Neumann boundary conditions are imposed on the curved lateral surface, 
while the top and bottom planar faces are treated as Dirichlet boundaries. 

To perform the convergence analysis, 
we construct a sequence of four meshes by extruding a Voronoi tessellation of a circle along the $z$-axis. 
In Figure~\ref{fig:cyliVoro}, we show one of these meshes along with a clipped view, 
taken with respect to a plane perpendicular to the $xOy$ plane,
to highlight the extrusion process used to generate the mesh. 
We consider VEM approximation degrees $k=1$ and $k=2$.

\begin{figure}[!htb]
\centering
\includegraphics[width=0.8\textwidth]{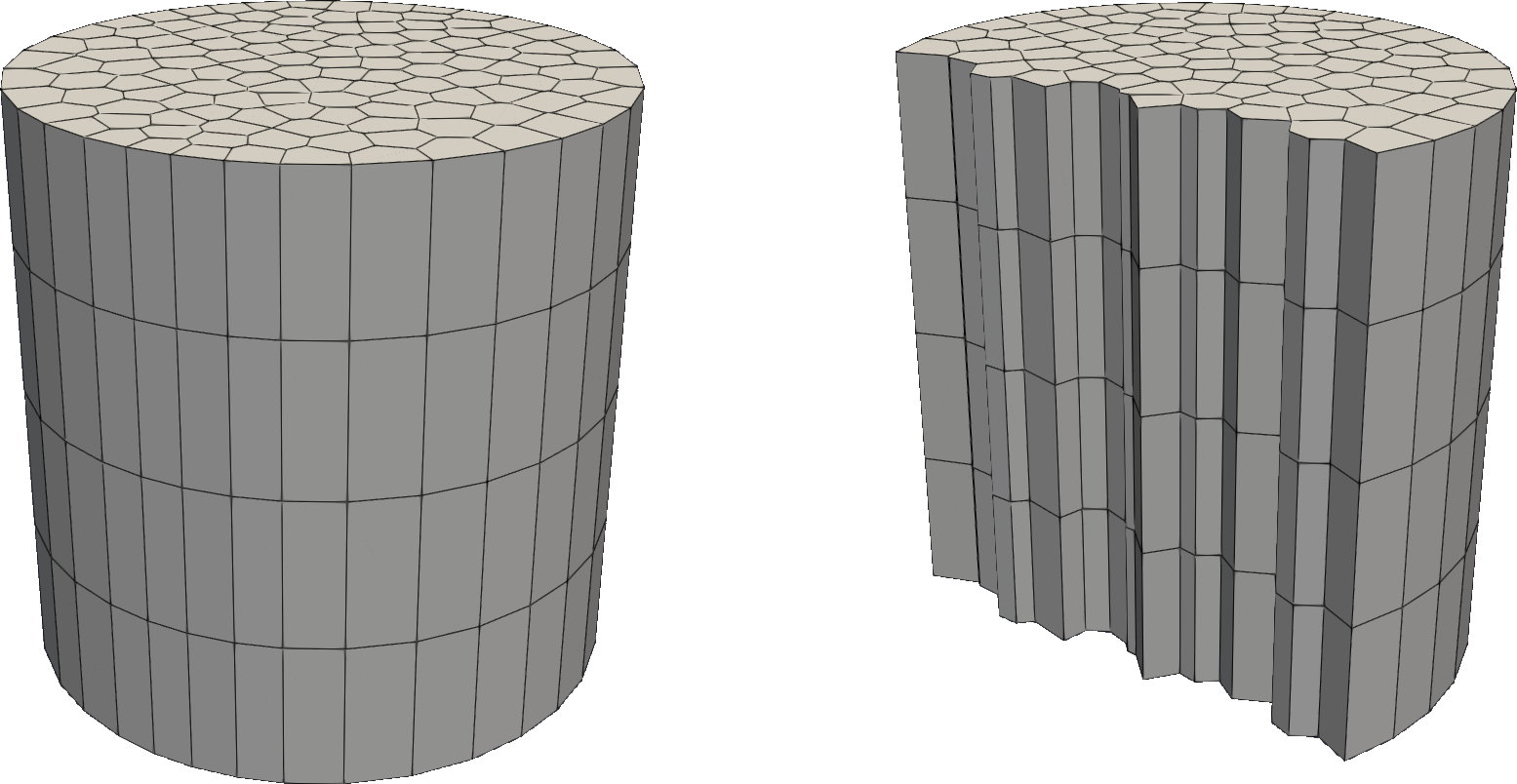}
\caption{The full cylindrical mesh, on the left, and a clipped view right highlighting the internal extruded structure, on the right.}
\label{fig:cyliVoro}
\end{figure}

It is worth noting that the resulting meshes feature curved quadrilateral faces on the lateral surface. 
From an implementation standpoint, the curved geometry of the cylinder is handled effectively; 
indeed, all required geometrical quantities, such as surface normals or tangents,
are computed via an exact parametrization of the quadrilateral faces, 
defined on the reference unit square $[0,1]^2$.

Here, the global discrete space \(\mathbf{V}^k_h(\Omega_h)\) is built in such a way that 
all quadrilaterals lying on the lateral surfaces of the cylinder are non-conforming faces,
while the remaining faces are conforming.

The convergence plots of $e_{L^2}$ and $e_{H^1}$ for VEM approximation degrees $k=1$ and $2$ are reported in Figure \ref{fig:test2-results}. 
Even in this curved case, where the domain geometry is exactly represented, the convergence rates for all error indicators remain optimal; specifically, we observe $\mathcal{O}(h^{k+1})$ for the $L^2$-error and $\mathcal{O}(h^k)$ for the $H^1$-error. As a comparison, note that the expected error rates for a second-order finite element scheme ($k=2$) based on a tetrahedral mesh family, interpolatory at the domain boundary, would be $\mathcal{O}(h^{2})$ and $\mathcal{O}(h^{3/2})$, respectively.

\begin{figure}[htbp]
\centering
\begin{tabular}{cc}
\includegraphics[width=0.42\textwidth]{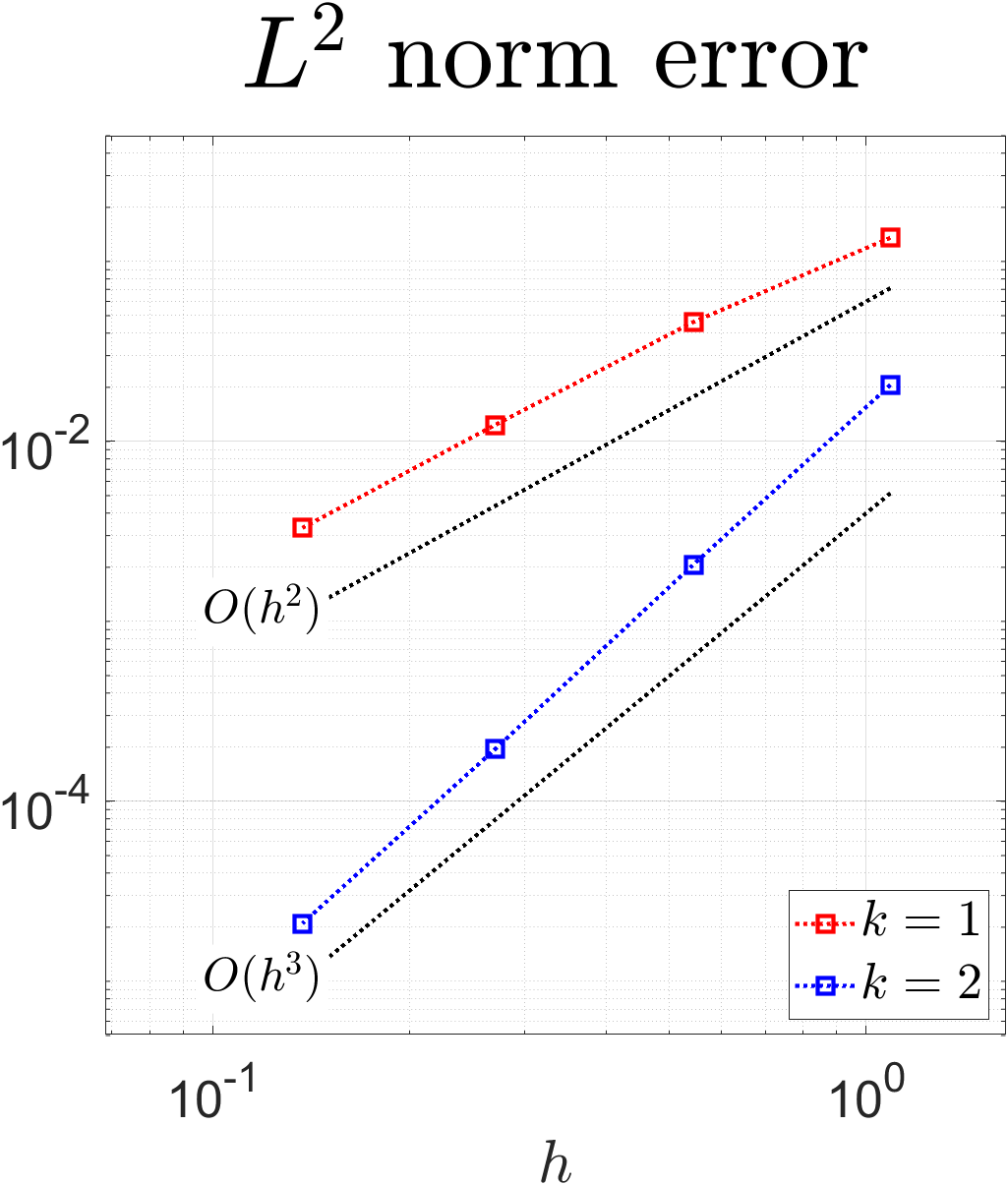}  &
\includegraphics[width=0.42\textwidth]{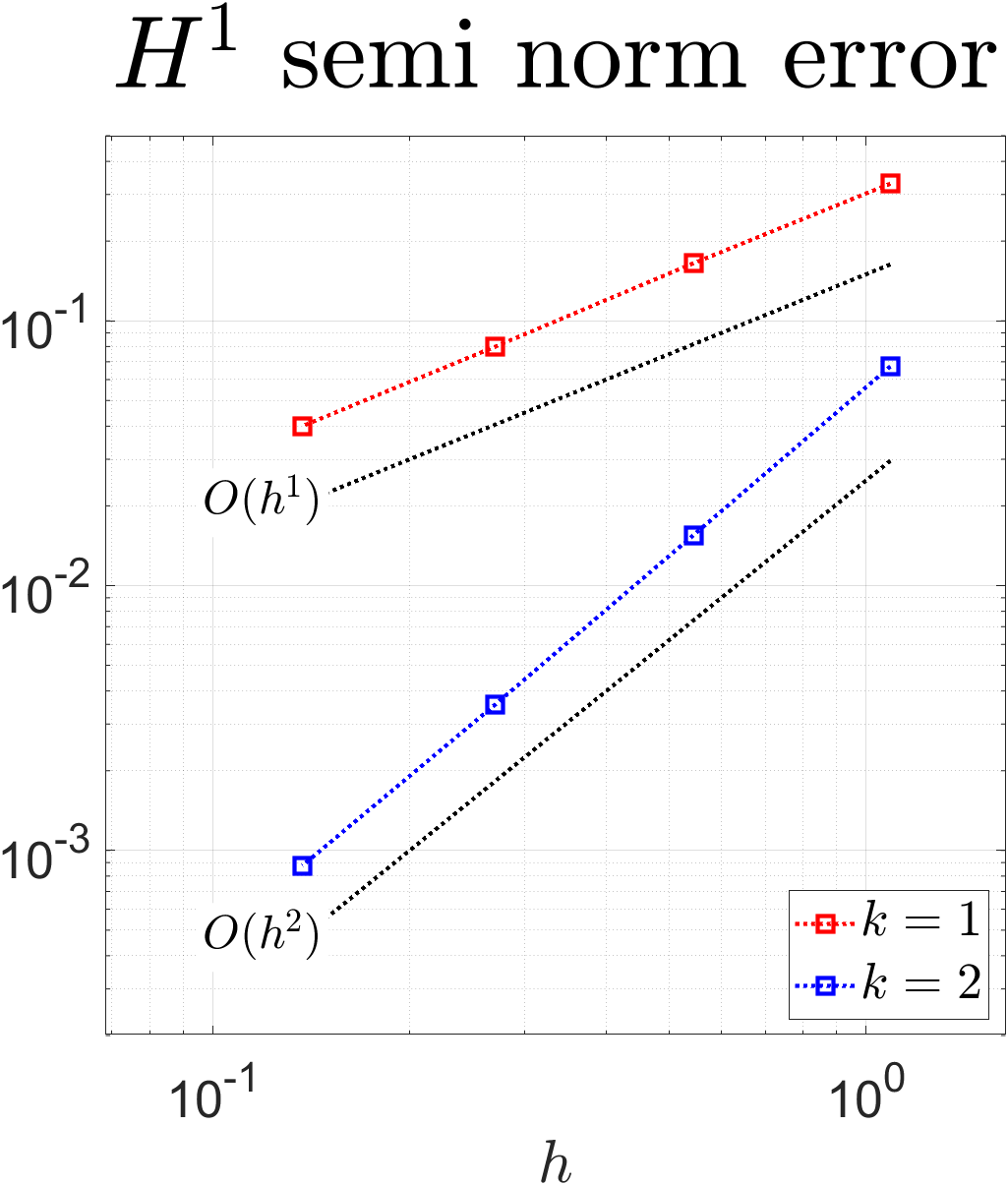}
\end{tabular}
\caption{Planar and curved faces: convergence lines of $e_{L^2}$ and $e_{H^1}$.}
\label{fig:test2-results}
\end{figure}

\subsection{Curved boundary tessellated with arbitrary polygons}\label{sec:convGenCur}

In this section, we provide a proof of concept of the applicability of the proposed method 
to meshes where the curved boundary is tessellated with arbitrary polygons.

Let $\Omega$ be the sphere centered at the origin with radius 1.
We set $\lambda = \mu = 1$ and the data of the problem in such a way that 
the exact displacement field is given by 
\[
\uu(x, y, z) = 
(1 - x^2-y^2-z^2)
\begin{pmatrix}
e^x \\
e^y \\
e^z
\end{pmatrix}.
\]
We consider fully homogeneous Dirichlet boundary conditions on the entire boundary, i.e.,
$\Gamma_N=\emptyset$, $\Gamma_D=\partial \Omega$ and $\uu|_{\partial \Omega} = \mathbf{0}$.

To generate a sequence of four meshes with decreasing mesh size, 
we start from a Voronoi tessellation produced via \texttt{voro++}~\cite{voro++},
then optimized through Lloyd's algorithm. 
Since \texttt{voro++} clips the Voronoi cells using planes to make a mesh constrained to a sphere, 
the resulting boundary faces are planar but discontinuous, see Figure~\ref{fig:VORO++} where we enhanced the gaps for a better vizualization. 
\begin{figure}[!htb]
\centering
\includegraphics[width=0.40\textwidth]{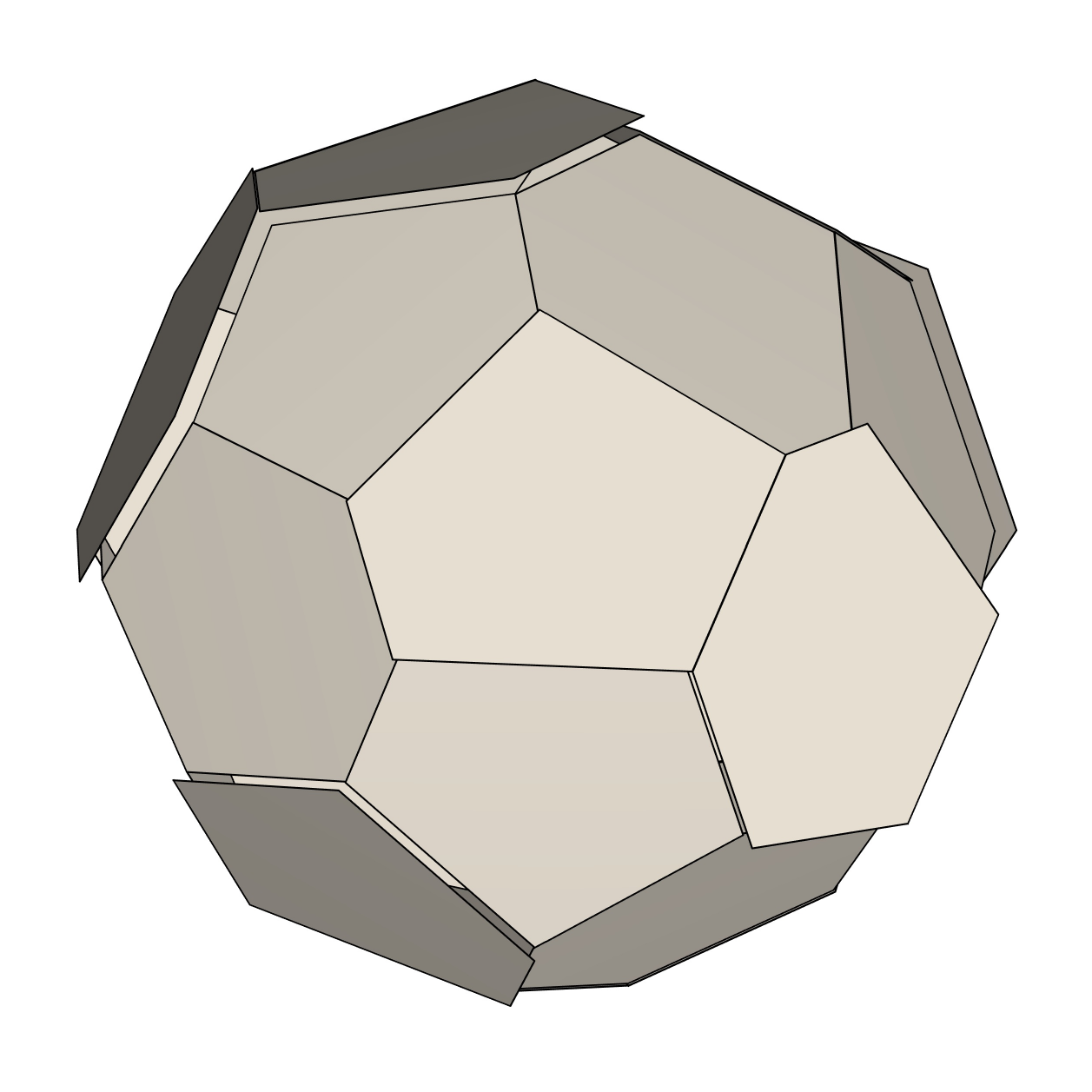}
\caption{The output of \texttt{voro++} (with enhanced gaps for better visualization)}
\label{fig:VORO++}
\end{figure}
To ensure that the boundary vertices accurately represent the geometry,
we project them onto the unit sphere. 
To approximate the spherical surface over these polygons and solve the PDE, 
we employ a mapping approach where each map is constructed for each element based on the number of edges of the physical polygon, i.e.,
given a physical polygon with $n_e$ edges, 
we define a map starting from a reference regular polygon with exactly the same number of vertices. 
Within this framework, we define two different approaches,
see Fig. \ref{fig:LIN+QUAD}:
\begin{itemize}
\item \textit{Linear approximation of geometry} (LIN): the computational surface is defined by interpolating the $n_e$ vertices of the polygon using Wachspress functions.
This results in a non-planar polygon with straight edges;
\item \textit{Quadratic approximation of geometry} (QUAD):
for each straight edge of the polygon, the midpoint is projected onto the sphere, 
and an arc of a parabola is constructed to connect the vertices through this projected midpoint.
Then, the interior is approximated by combining Wachspress functions to interpolate these parabolic arcs. 
This construction is based on the second-order serendipity basis functions introduced in \cite{shavelikeabomber}.
\end{itemize}
\begin{figure}[!htb]
\centering
\includegraphics[width=0.35\textwidth]{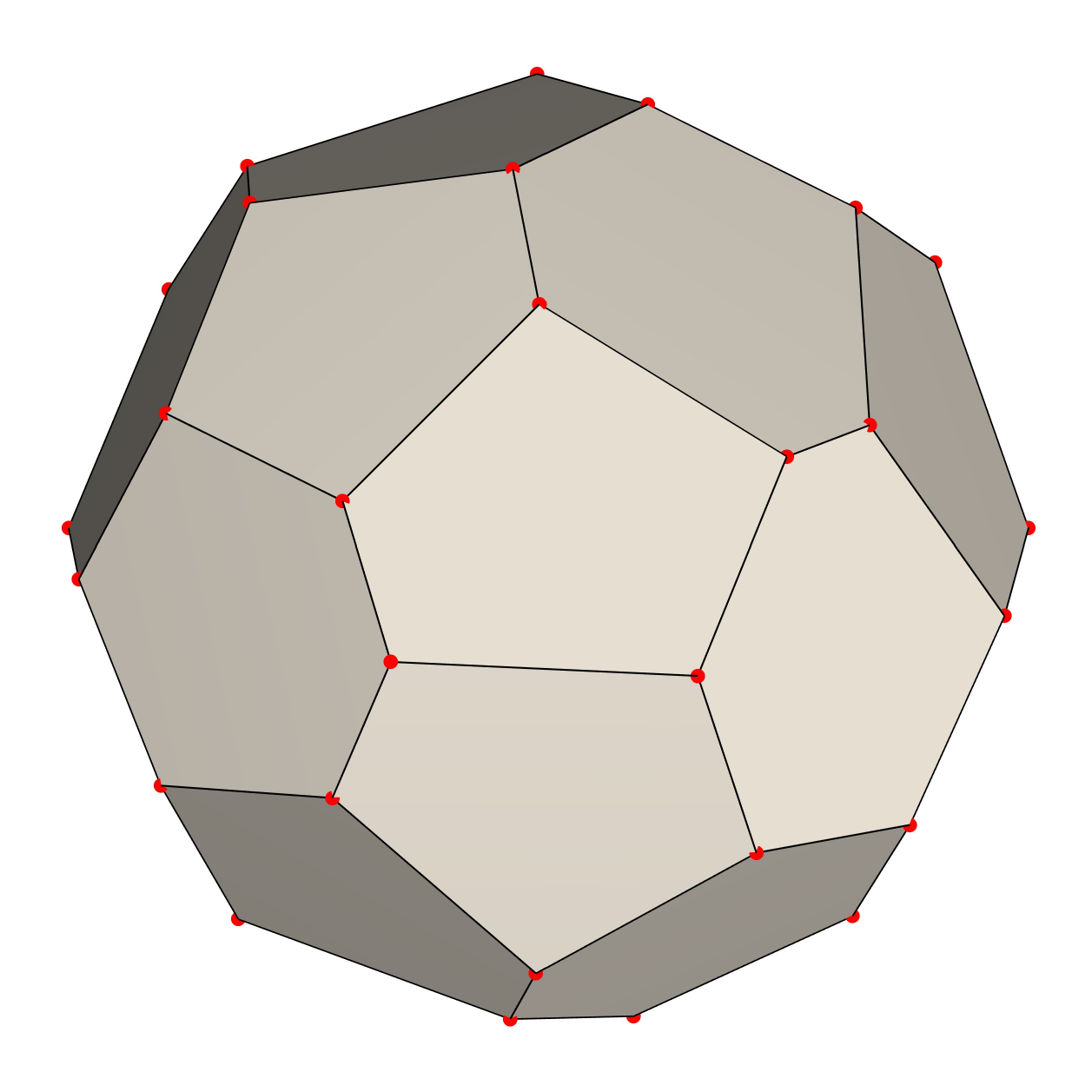}
\includegraphics[width=0.35\textwidth]{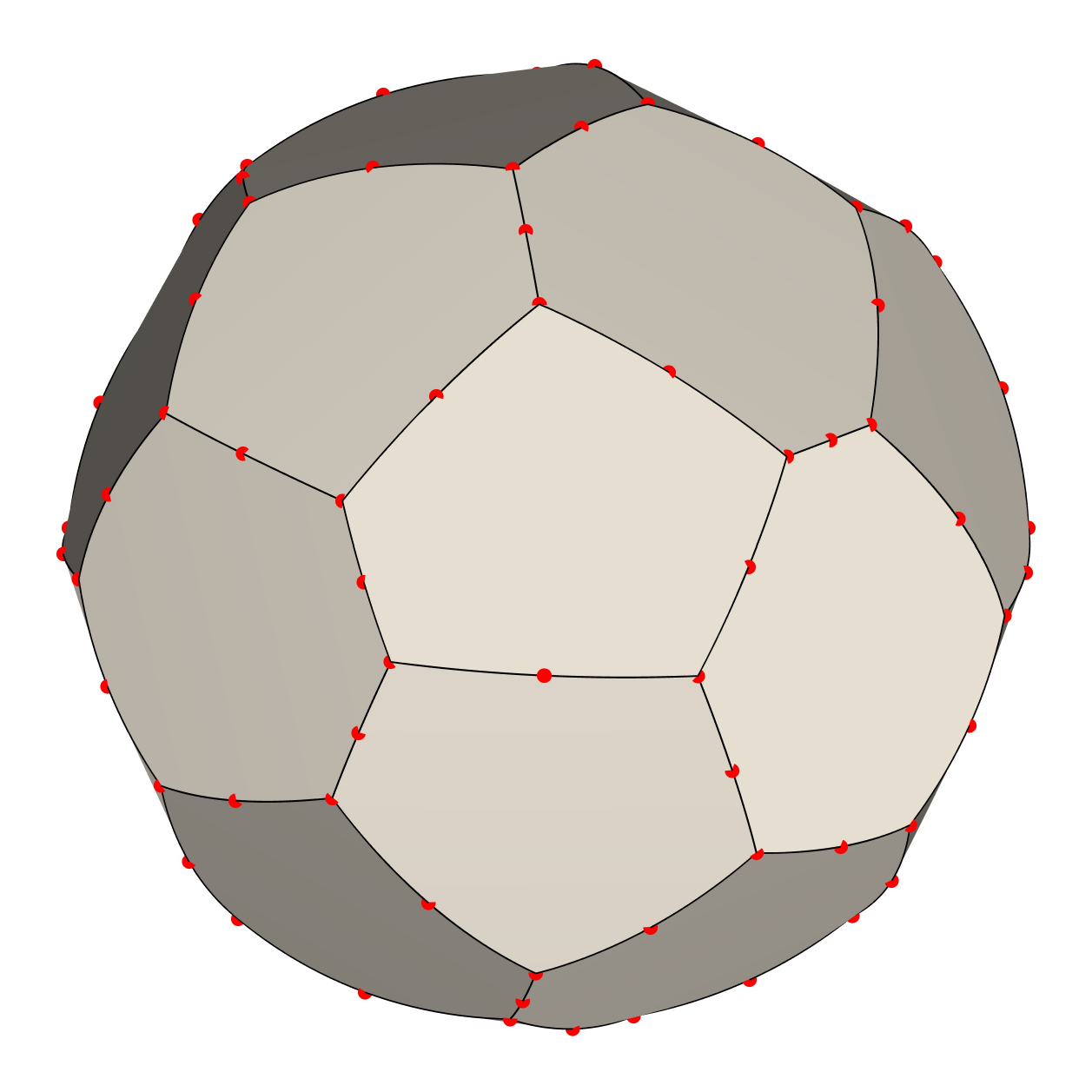}
\caption{The geometrical approximation LIN (left) and QUAD (right). All red points lie exactly on the unit sphere.}
\label{fig:LIN+QUAD}
\end{figure}
The global discrete space $\mathbf{V}^k_h(\Omega_h)$ is constructed such that all faces lying on $\partial \Omega$ are non-conforming,
while the remaining internal faces are conforming.
We analyze the convergence behavior for VEM degrees $k=1$ and $k=2$, 
approximating the spherical surface using both the Linear and Quadratic geometry approaches. 
Specifically, we consider the following combinations:
\begin{itemize}
\item $k=1$ with linear geometry (Lin),
\item $k=2$ with linear geometry (Lin),
\item $k=2$ with quadratic geometry (Quad).
\end{itemize}

From a theoretical perspective, 
we expect the following behaviors which are confirmed by the numerical results reported in Figure~\ref{fig:test3-results}.

Specifically, for $k=1$ with linear geometry, 
we recover the optimal convergence rates for both $e_{L^2}$ and $e_{H^1}$. 
Indeed, in this case, the geometric error carries sufficient accuracy to maintain the optimal rates of the VEM solution.

On the contrary, for $k=2$, the optimal convergence rates for both $e_{L^2}$ and $e_{H^1}$ are recovered only 
if the curved geometry is approximated with sufficient accuracy.
This fact is confirmed by the numerical experiments: 
the quadratic geometry leads to much lower error values, especially for the $L^2$ error norm which initially exhibits an error reduction rate ${\mathcal O}(h^3)$ instead of ${\mathcal O}(h^2)$ as for the linear geometry approximation. 
However, the rate of convergence for $k=2$ with quadratic geometry plateaus in the final iterations.
This behavior is likely due to the geometric approximation error becoming dominant over the PDE discretization error and is related to the fact that the serendipity approximation adopted in the quadratic approach does not guarantee an ${\mathcal O}(h^3)$ geometric error. 
To further validate this observation, Figure \ref{fig:geoerror} plots the surface area and volume errors between the exact and approximated geometries. It is evident that, even with a quadratic (Quad) approximation, these errors do not decay at a rate of $O(h^3)$. This is significant, as $O(h^3)$ is the required order to achieve the optimal $L^2$ convergence rate for displacement when using second-order ($k=2$) elements.

Different approaches, such as an exact parametric representation of the spherical surface or using a graph-like approach, could cure such issue. Since the focus of the present experiments is to showcase some simple example of application of our conforming-nonconforming VEM element, rather than investigating all its possibilities, exploring such directions is beyond the scope of this article.

\begin{figure}[htbp]
\centering
\begin{tabular}{cc}
\includegraphics[width=0.42\textwidth]{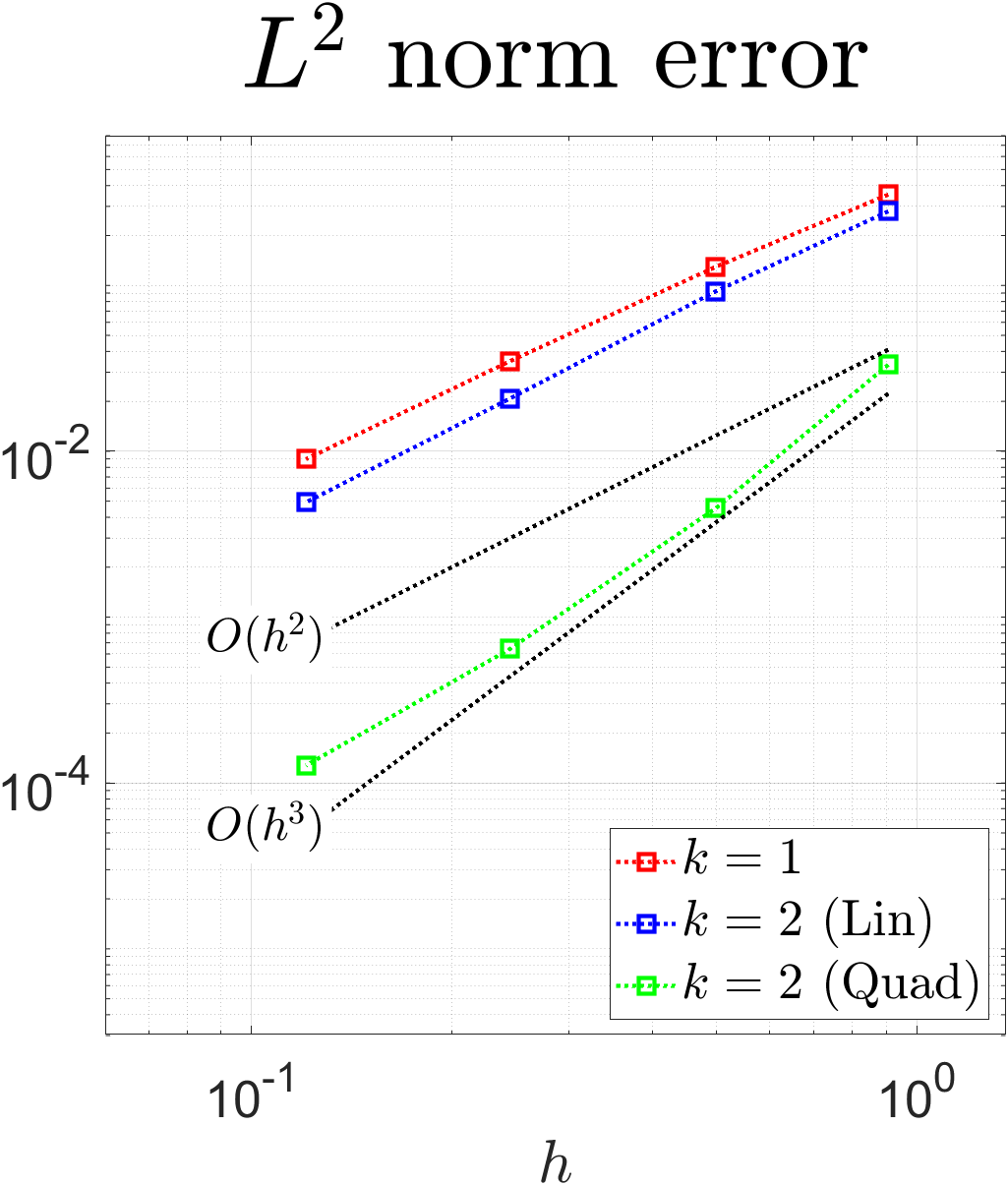}  &
\includegraphics[width=0.42\textwidth]{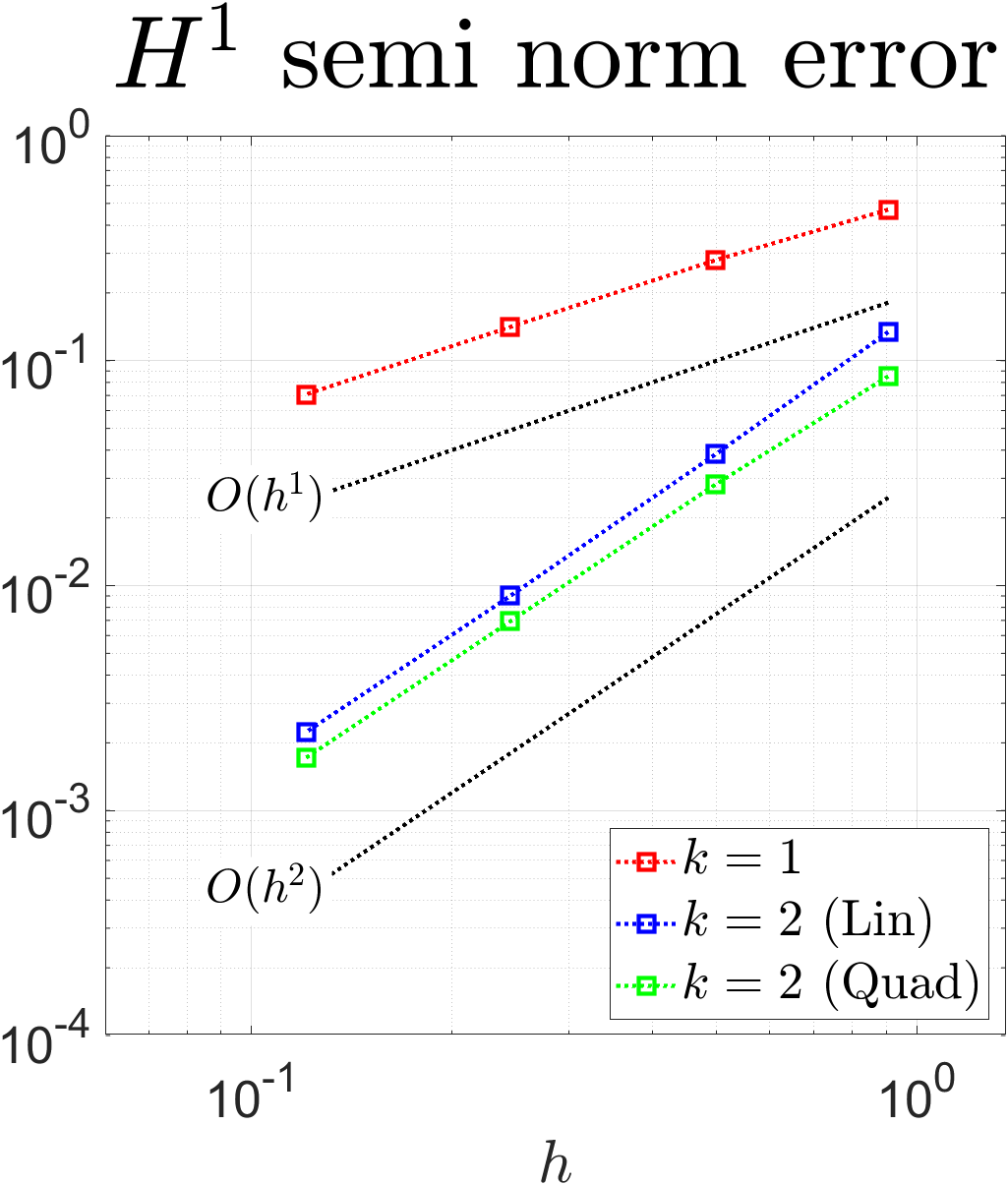}
\end{tabular}
\caption{Curved boundary tessellated with arbitrary polygons: convergence lines of $e_{L^2}$ and $e_{H^1}$.}
\label{fig:test3-results}
\end{figure}

\begin{figure}[htbp]
\centering
\begin{tabular}{cc}
\includegraphics[width=0.42\textwidth]{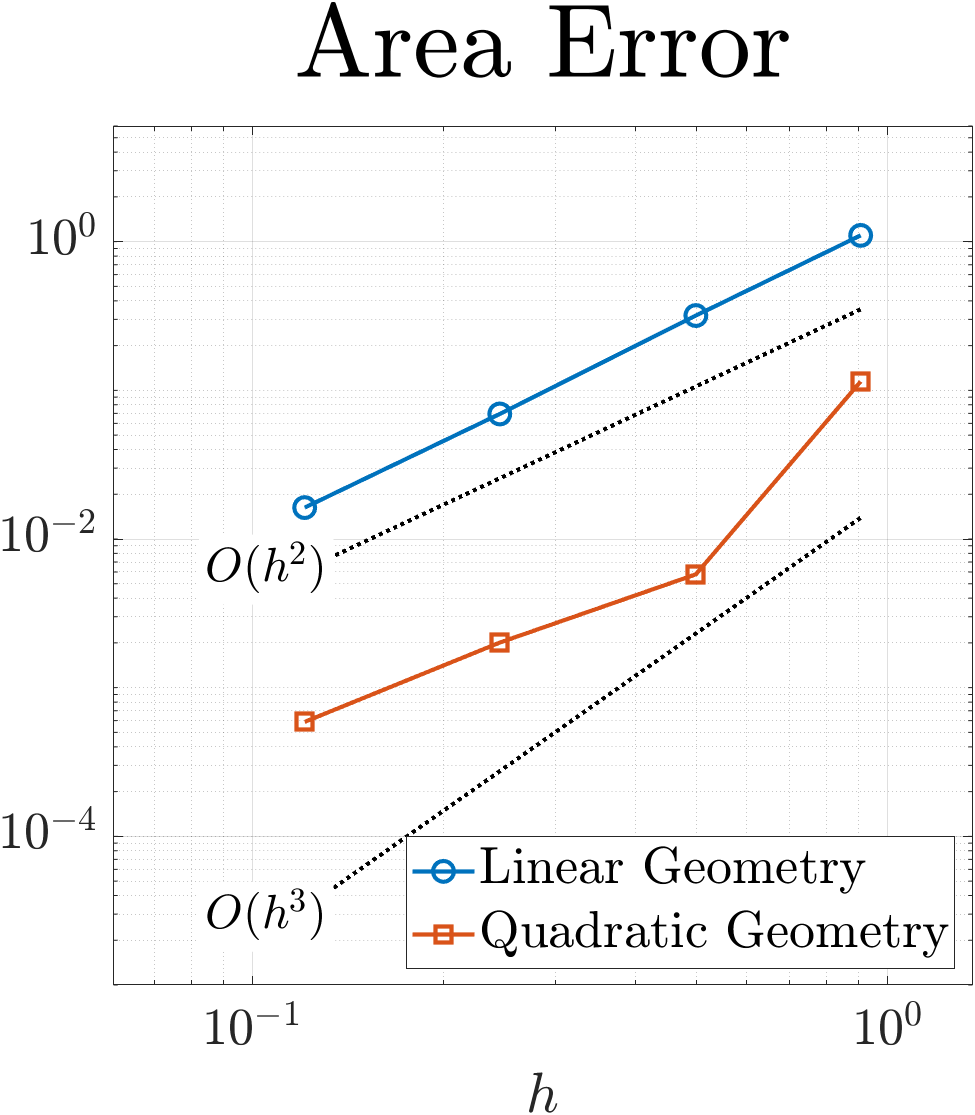}  &
\includegraphics[width=0.42\textwidth]{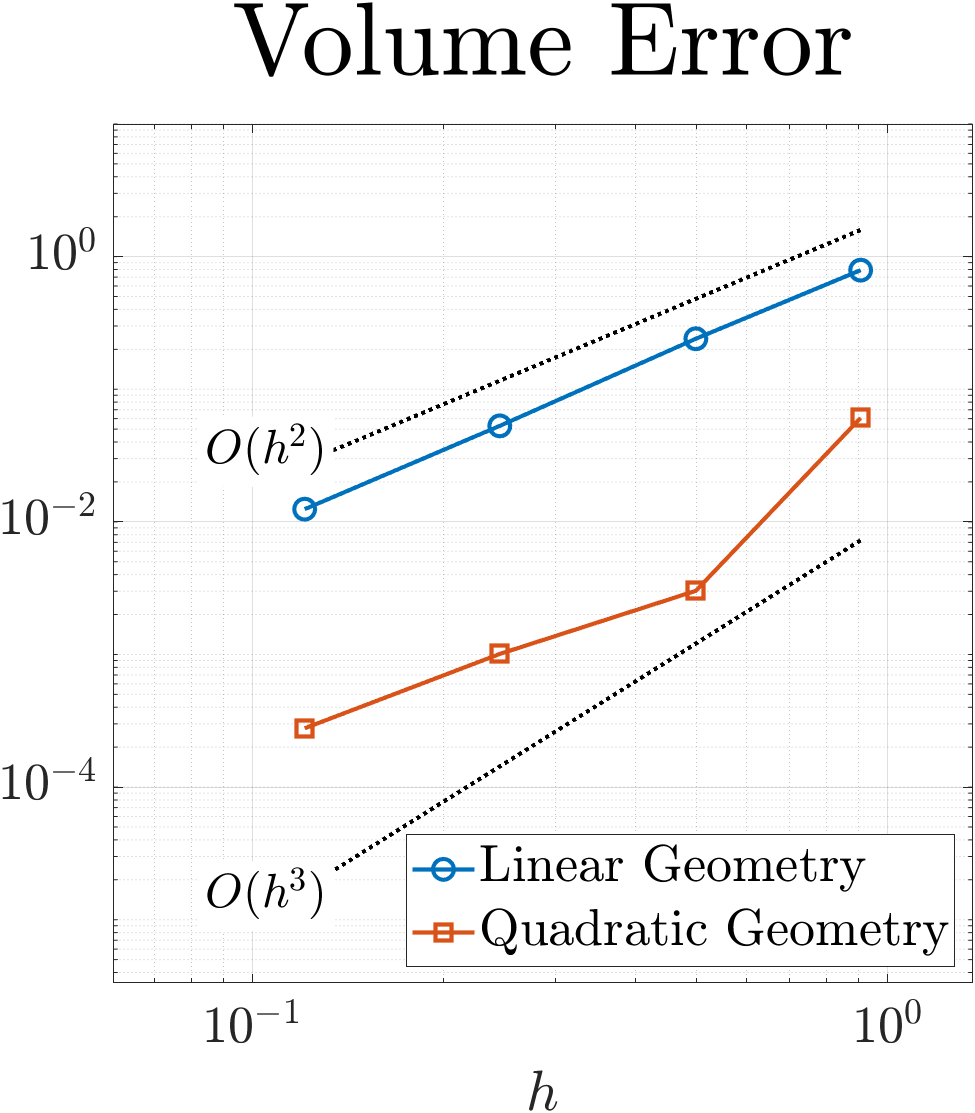}
\end{tabular}
\caption{Error in approximating the surface area and the volume of the sphere.} 
\label{fig:geoerror}
\end{figure}

\smallskip
\begin{center}
{\bf Aknowledgements}
\end{center}

\smallskip\noindent
LBDV, FD and MT has been partially funded by the European Union (ERC Synergy, NEMESIS, project number 101115663).
Views and opinions expressed are however those of the authors only and do not necessarily reflect those of the European Union or the ERC Executive Agency. 
All the authors are members of the INdAM Research Group GNCS and acknowledge partial support from INdAM-GNCS.

\bibliographystyle{plain}
\bibliography{biblio2}
\end{document}